%% file: main.tex
\documentclass[twoside,a4paper]{article}
\include{style}

\title{Efficient Learning of a Linear Dynamical System with Stability Guarantees}
\author{Wouter Jongeneel, Tobias Sutter and Daniel Kuhn
\thanks{Wouter Jongeneel and Daniel Kuhn are with the Risk Analytics and Optimization Chair, École Polytechnique Fédérale de Lausanne, \{wouter.jongeneel, daniel.kuhn\}@epfl.ch. Tobias Sutter is with the Department of Computer Science, University of Konstanz, tobias.sutter@uni-konstanz.de. This research was supported by the Swiss National Science Foundation under the NCCR Automation, grant agreement~51NF40\_180545.}}

\begin{document}
\maketitle
\thispagestyle{empty}

\begin{abstract}
We propose a principled method for projecting an arbitrary square matrix to the non-convex set of asymptotically stable matrices. Leveraging ideas from large deviations theory, we show that this projection is optimal in an information-theoretic sense and that it simply amounts to shifting the initial matrix by an optimal linear quadratic feedback gain, which can be computed exactly and highly efficiently by solving a standard linear quadratic regulator problem. The proposed approach allows us to learn the system matrix of a stable linear dynamical system from a single trajectory of correlated state observations. The resulting estimator is guaranteed to be stable and offers statistical bounds on the estimation error.
\end{abstract}

\section{Introduction}
\label{sec:introduction}
We study the problem of learning a stable linear dynamical system from a single trajectory of correlated state observations. 
This problem is of fundamental importance in various disciplines such as adaptive control \autocite{ref:WittenmarkAstrom_73}, system identification \autocite{ref:Kumar-86,Verhaegen}, reinforcement learning \autocite{ref:Sutton-1998,ref:BerDim-19,ref:SelfTuning, ref:recht2019tour} and approximate dynamic programming \autocite{ref:BerTsi-96, powell2007approximate}. Specifically, we consider a discrete-time linear time-invariant system of the form
\begin{equation} \label{eq:LTI:system}
x_{t+1} = \theta x_t + w_t, \quad x_0\sim \nu,
\end{equation}
where $x_t\in\mathbb{R}^n$ and $w_t\in\mathbb{R}^n$ denote the state and the exogenous noise at time $t\in\mathbb{N}$, respectively, while~$\theta$ represents a {fixed} system matrix, and $\nu$ stands for the marginal distribution of the initial state~$x_0$. We assume that $\theta$ is \textit{asymptotically stable}, that is, it belongs to $\Theta=\{{\theta}\in \mathbb{R}^{n\times n}:\rho({\theta})<1\}$, where $\rho({\theta})$ denotes the largest absolute eigenvalue of ${\theta}$. For ease of terminology, we will usually refer to $\Theta$ as the set of \textit{stable} matrices and to its complement in $\mathbb{R}^{n\times n}$ as the set of \textit{unstable} matrices. We assume that nothing is known about $\theta$ except for its membership in~$\Theta$, and we aim to learn $\theta$ from a single-trajectory of data $\{\widehat{x}_t\}_{t=0}^T$ generated by~\eqref{eq:LTI:system}. To this end, one can use the least squares estimator
\begin{equation} \label{eq:LS:MDP}
\widehat \theta_T = \left(\textstyle\sum_{t=1}^T \widehat x_t \widehat x_{t-1}^\mathsf{T}\right)\left(\textstyle\sum_{t=1}^T \widehat x_{t-1}\widehat x_{t-1}^\mathsf{T}\right)^{-1},
\end{equation}
which may take any value in $\Theta'=\mathbb{R}^{n\times n}$ under standard assumptions on the noise distribution. It is therefore possible that $\widehat{\theta}_T\notin \Theta$ even though $\theta\in \Theta$. 
This is troubling because stability is important in many applications, for example, when the estimated model is used for prediction, filtering or control, e.g., see the discussions in~\cite[pp.~53--60, 125--129]{ref:OverscheeDeMoor1996}. Estimating a stable model is also crucial for assessing the performance of a stable system or, in a simulation context, for generating useful exploration data.  


Given the prior structural information that~$\theta$ is stable, we thus seek an estimator that is guaranteed to preserve stability. A natural approach to achieve this goal would be to project the least squares estimator $\widehat{\theta}_T$ to the nearest stable matrix with respect to some discrepancy function on~$\mathbb{R}^{n\times n}$. This seems challenging, however, because $\Theta$ is open, unbounded and non-convex; see Figure~\ref{fig:2D:ncvx}($a$). To circumvent this difficulty, we introduce a new discrepancy function that adapts to the geometry of $\Theta$ and is thus ideally suited for projecting unstable matrices onto~$\Theta$. We will characterize the statistical properties of this projection when applied to the least squares estimator, and we will show that it can be computed efficiently even for systems with $ O(10^3)$ states.

The following example shows that na\"ive heuristics to project~$\theta'$ into the interior of~$\Theta$ could spectacularly fail.
\begin{example}[Projection by eigenvalue scaling]
\label{ex:proj}
\upshape{ A na\"ive method to stabilize a matrix~$\theta'\notin\Theta$ would be to scale its unstable eigenvalues into the complex unit circle. To see that the output of this transformation may not retain much similarity with the input~$\theta'$, consider the matrices
\begin{equation*}
    {\theta}'\! =\! \begin{bmatrix}
    1.01 & 10\\ .01 & 1
    \end{bmatrix},\; {\theta}'_a\!=\!\begin{bmatrix}
    .84 & 4.77 \\ .005 & .84
    \end{bmatrix},\; {\theta}'_b\! =\! \begin{bmatrix}
    .99 & 10 \\ 0 & .99
    \end{bmatrix}.
\end{equation*}
Clipping off the unstable eigenvalues of $\theta'$ at $|\lambda|=.99$ yields ${\theta}'_a$ with $\rho(\theta'_a)=.99$ and $\|{\theta}'-{\theta}'_a\|_2\gtrsim 5$. However, the matrix $\theta'_b$ also has spectral radius $\rho(\theta'_b)=.99$ but is much closer to $\theta'$. Indeed, we have $\|{\theta}'-{\theta}'_b\|_2\approx 0.02$. 
}
\end{example}

\begin{figure}[t!]
    \centering
    \includegraphics[angle=270,scale=0.6]{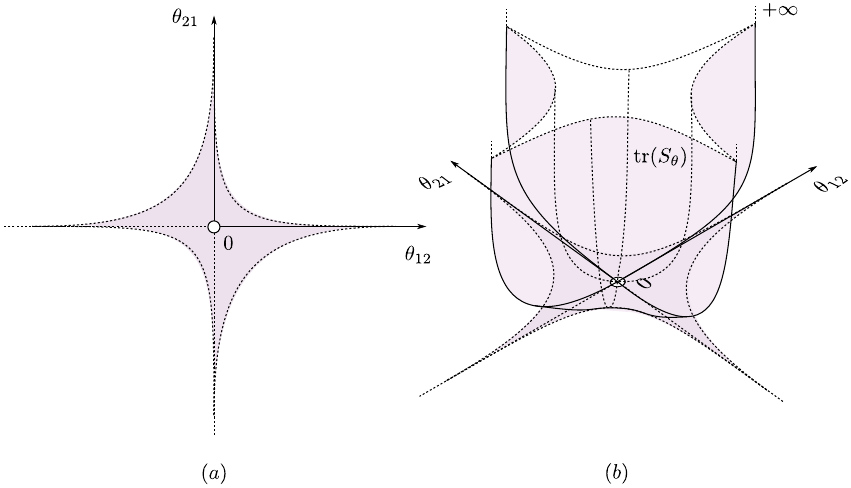}
    \caption{$(a)$ Visualization of a two-dimensional slice of the set~$\Theta$ of stable $2\times 2$-matrices with diagonal elements~$\theta_{11}=\theta_{22}=\frac{1}{2}$. The set of all corresponding off-diagonal elements is open, non-convex and unbounded along the coordinate axes. $(b)$ Some eigenvalues and the trace of the covariance matrix~$S_\theta$ of the invariant state distribution explode as~$\theta$ approaches the boundary of~$\Theta$.}
    \label{fig:2D:ncvx}
\end{figure}



\subsection{Related work}\label{ssec:related:work}
The problem of learning a stable dynamical system is widely studied in system identification, while the problem of projecting an unstable matrix onto~$\Theta$ with respect to some norm has attracted considerable interest in matrix analysis. 

In the context of identification, \textcite{ref:Maciejowski:stab} proposed one of the first methods to project a possibly unstable estimator onto $\Theta$ by using subspace methods. This pioneering approach has significant practical merits \autocite{ref:OverscheeDeMoor1996} but may also significantly distort the original estimator.
To overcome this deficiency, \textcite{ref:LacyBernstein2002} approximate $\Theta$ by the set of \textit{contractive} matrices whose operator norm is at most~$1$. While this set is convex, it offers but a conservative approximation of~$\Theta$. Several related methods have since been proposed to enforce stability~\autocite{ref:LacyBernstein2003,ref:Boots:Nips,ref:stab:Gersh}, which are all either conservative or computationally expensive. Moreover, these methods do not provide any statistical guarantees. \textcite{ref:vanGestel1,ref:vanGestel2} regularize the least squares objective and show that the spectral radius of the resulting estimator is bounded by a function of the regularization weights. As~$\Theta$ is an open set, however, the tuning of these weights remains a matter of taste. 
More recently, \textcite{ref:stableEM} propose a maximum likelihood approach that is attractive from a statistical point of view but can be computationally challenging in certain applications. 
On the other hand, several authors use Lyapunov theory to provide stability guarantees for deterministic vector fields; see, {\em e.g.}, \autocite{ref:Khansari-Zadeh,ref:Berkenkamp2017SafeRL,ref:KolterNips2019,ref:Umlauft}. A more recent approach by \textcite{ref:boffi2020learning} learns stability certificates from i.i.d.\ trajectories.
There is also a substantial body of literature on (sub-)optimal finite-sample concentration bounds for linear systems identified via least squares estimation \autocite{ref:Sim_18, ref:jedra-19,ref:sarkar19a, ref:Jedra-20, ref:sarkar-20nonparametric}. These approaches offer fast learning rates but cannot guarantee stability of the identified systems for finite sample sizes. 


Much like in dynamical systems theory, in matrix analysis one seeks algorithms for projecting an unstable deterministic matrix $\theta'$ onto $\Theta$, which is equivalent to finding the smallest additive perturbation that stabilizes~$\theta'$. More specifically, matrix analysis studies the \textit{nearest stable matrix problem}
\begin{equation}
    \label{equ:standard:proj}
    \Pi_{\Theta}(\theta') \in \arg \min_{\theta\in \mathop{\mathsf{cl}} \Theta} \|\theta' - \theta \|^2,
\end{equation}
where $\|\cdot \|$ represents a prescribed norm on~$\mathbb R^{n\times n}$. Note that optimizing over the closure of~$\Theta$ is necessary for~\eqref{equ:standard:proj} to be well-defined because any minimizer lies on the boundary of the open set~$\Theta$.
Solving \eqref{equ:standard:proj} is challenging because $\Theta$ is non-convex. Existing numerical solution procedures rely on successive convex approximations \autocite{ref:Orbandexivry-13}, on local optimization schemes based on the solution of low-rank matrix differential equations
\autocite{ref:Guglielmi-17} or on an elegant reparametrization of the set of stable matrices, which simplifies the numerics of the projection operation \autocite{ref:Gillis-19,ref:Choudhary-20}. The latter approach was recently used for learning stable systems~\autocite{ref:MamakoukasGillisGrad,ref:Koopman2020}.  
\textcite{ref:NestNonNegMat} solve~\eqref{equ:standard:proj} for certain polyhedral norms and non-negative matrices~$\theta'$, which allows them to find exact solutions. See~\autocite{ref:higham1988matrix} for a general discussion on matrix nearness problems.  

Optimal control offers a promising alternative perspective on problem~\eqref{equ:standard:proj}, which is closely related to the approach advocated in this paper: one could try to design a linear quadratic regulator (LQR) problem whose optimal feedback gain $K^{\star}\in \mathbb{R}^{n\times n}$ renders $\theta'+K^{\star}$ stable. By proposing an LQR objective that is inversely proportional to the sample covariance matrix of the measurement noise, \textcite{ref:Tanaka} show that this idea is indeed valid, but they provide no error analysis or statistical guarantees. 
Using optimal control techniques, \textcite{ref:Topo2020} prove that one can find matrices $K$ that not only render $\theta'+K$ stable but are also structurally equivalent to $\theta'$ (\textit{e.g.,} $\theta'+K$ preserves the null space of $\theta'$). Such a structure-preserving approach seems preferable over the plain nearest stable matrix problem~\eqref{equ:standard:proj}, which merely seeks stability at minimal `cost'. Appealing to the theory of \textit{large deviations}, we will give such approaches a statistical underpinning.

\textbf{Notation.}
For a matrix $A\in\mathbb{C}^{n\times n}$, we denote by $\rho(A)$ the largest absolute eigenvalue 
and by $\kappa(A)$ 
the condition number of $A$. 
For a set $\mathcal{D}\subset \mathbb{R}^n$, we denote by $\mathcal{D}^{\mathsf{c}}$ the complement, by $\mathsf{cl} \, \mathcal{D}$ the closure and by $\mathsf{int}\, \mathcal{D}$ the interior of~$\mathcal{D}$. For a real sequence $\{a_T\}_{T\in\mathbb{N}}$ we use $1\ll a_T\ll T$ to express that $a_T /T \to 0$ and $a_T\to\infty$ as $T\to\infty$. {We also use the soft-$O$ notation $\tilde O(f(T))$ as a shorthand for $O(f(T)\log(T)^c)$ for some~$c\in\mathbb N$, that is, $\tilde O(\cdot)$ ignores polylogarithmic factors.} 

\newpage
\subsection{Contributions}
\label{sec:contributions}
{Throughout the paper we assume that all random objects are defined on a measurable space~$(\Omega, \mathcal{F})$ equipped with a probability measure~$\mathbb{P}_\theta$ that depends parametrically on the \textit{fixed} yet unknown system matrix~$\theta$, and the system equations~\eqref{eq:LTI:system} are assumed to hold~$\mathbb P_\theta$-almost surely; see also the discussion below Assumption~\ref{ass:linear:sys}. The expectation operator with respect to $\mathbb{P}_\theta$ is denoted by $\mathbb{E}_{\theta}[\cdot]$.} Even though the least squares estimator~$\widehat \theta_T$ is strongly consistent and thus converges $\mathbb{P}_\theta$-almost surely to~$\theta$ \autocite{ref:Campi_98}, it differs $\mathbb{P}_\theta$-almost surely from~$\theta$ for any finite~$T$. To quantify estimation errors, we introduce a discrepancy function $I:\Theta'\times \Theta\rightarrow [0,\infty]$ defined through
\begin{equation} \label{eq:rate:function:AR:MDP}
	I(\theta^\prime,  \theta) = \frac{1}{2} \mathsf{tr}\left(S_w^{-1}(\theta^\prime - \theta) S_\theta (\theta^\prime - \theta)^\mathsf{T}\right).
\end{equation}
Here,~{$S_w\succ 0$ stands for the time-independent noise covariance matrix,} and~$S_\theta
$ denotes the covariance matrix of $x_t$ under the stationary state distribution, which exists for~$\theta\in\Theta$ but diverges as~$\theta$ approaches the boundary of $\Theta$; see Figure~\ref{fig:2D:ncvx}($b$). Note that since $S_w\succ 0$ and hence $S_{\theta}\succ 0$, $I(\theta^\prime,  \theta)$ vanishes if and only if $\theta'=\theta$. In this sense $I$ behaves like a distance.
Note, however, that $I(\theta^\prime,  \theta)$ is not symmetric in~$\theta$ and~$\theta'$.
In this paper we propose to use the discrepancy function~\eqref{eq:rate:function:AR:MDP} for projecting an unstable matrix~$\theta'$ onto~$\Theta$. Specifically, we define the \textit{reverse $I$-projection} of any~$\theta'\in\mathbb R^{n\times n}$ as 
\begin{equation}
\label{equ:rev:I:proj}
    \mathcal{P}(\theta') \in \arg\min_{\theta\in\Theta} I(\theta',{\theta}).
\end{equation}
We emphasize that the minimum in~\eqref{equ:rev:I:proj} is always attained even though~$\Theta$ is open. The reason for this is that~$S_\theta$, and thus also~$I(\theta',\theta)$, diverges as $\theta$ approaches the boundary of~$\Theta$; see Proposition~\ref{prop:I_inf} below. Thus, the minimum must be attained inside~$\Theta$.
In fact, as~$I(\theta',{\theta})$ trades off distance against stability, $\mathcal{P}(\theta')$ may not even be close to the boundary of~$\Theta$; see Figure~\ref{fig:maptoI}.  
Moreover, we will see that the discrepancy function~\eqref{eq:rate:function:AR:MDP} has a natural statistical interpretation, which enables us to derive strong statistical guarantees for the reverse $I$-projection of the least squares estimator.

We will actually show that the discrepancy function~\eqref{eq:rate:function:AR:MDP} determines the speed at which the probability of the least squares estimator $\widehat \theta_T$ being sufficiently different from the true system matrix~$\theta$ decays with the sample size~$T$. 
Specifically, we will prove that the transformed estimator $\widehat \vartheta_T =\sqrt{T/a_T}(\widehat{\theta}_T-\theta)+\theta$ satisfies a moderate deviations principle with rate function~\eqref{eq:rate:function:AR:MDP}. By exploiting the relation $I(\widehat \theta_T,\theta)= (a_T/T)I(\widehat{\vartheta}_T,\theta)$, one can then show that the probability density function~$\varrho_{\theta,T}$ of the original least squares estimator $\widehat{\theta}_T$ with respect to the probability measure~$\mathbb P_\theta$ decays exponentially with~$T$, that is,
\begin{equation}
\label{equ:pdf}
    \varrho_{\theta, T}(\widehat{\theta}_T )
    \approx \mathrm{exp}(-I(\widehat{\theta}_T,\theta)\cdot T).
\end{equation}
Thus, the reverse $I$-projection $\mathcal{P}(\widehat{\theta}_T)$ maximizes the right-hand-side of~\eqref{equ:pdf} across all~$\theta\in\Theta$. Therefore, one can interpret $\mathcal{P}(\widehat{\theta}_T)$ as a \textit{maximum likelihood estimator}, that is, the most likely \textit{asymptotically stable model} in view of the data.

\begin{figure}
    \centering
    \includegraphics[angle=270,scale=0.6]{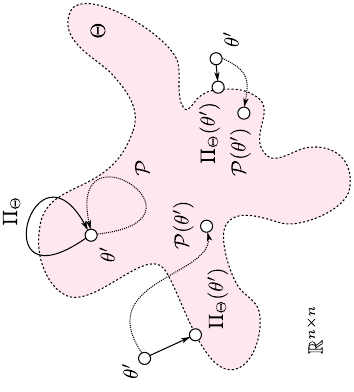}
    \caption{Schematic visualization of $\Pi_{\Theta}(\theta')$ and $\mathcal{P}(\theta')$ for different estimator realizations~$\theta'$ inside and outside of $\Theta$.}
    \label{fig:maptoI}
\end{figure}

Our main contributions can be summarized as follows.

\begin{enumerate}[(i)]
\item We prove that the discrepancy function~\eqref{eq:rate:function:AR:MDP} has a natural statistical interpretation as the rate function of a moderate deviation principle for the transformed least squares estimators~$\sqrt{T/a_T}(\widehat{\theta}_T-\theta)+\theta$, $T\in\mathbb N$.
\item We derive finite-sample and asymptotic statistical error bounds on the operator norm distance between the reverse $I$-projection $\mathcal{P}(\widehat{\theta}_T)$ of the least squares estimator~$\widehat \theta_T$ and the unknown true system matrix~$\theta$.
\item We show that the reverse $I$-projection $\mathcal{P}(\theta')$ can be computed highly efficiently to within any desired accuracy by solving a standard LQR problem, {\em e.g.}, via numerical routines that are readily available in MATLAB or Julia. This method finds the `cheapest' feedback gain matrix~$K^\star$ that renders~$\theta'+K^\star$ stable, and it can evaluate~$\mathcal{P}(\theta')$ in seconds even if~$n\approx 10^3$.
\end{enumerate}
In addition, numerical experiments corroborate our theoretical results and showcase the statistical and computational merits of using the reverse $I$-projection of $\widehat \theta_T$ to estimate~$\theta$. To our best knowledge, we present the first method for the identification of a linear dynamical system with stability guarantees that is both computationally efficient and offers asymptotic consistency and tight statistical error bounds.
The proposed method has been recently exploited to provide the first statistical result on \textit{qualitative} (topological) identification of a linear system~\cite{ref:topoID2021}. We also note that the derivation of the explicit rate function~\eqref{eq:rate:function:AR:MDP} is of independent interest in the context of statistical learning of linear dynamical systems.

\section{Main results} \label{sec:problem:statement}
From now on we impose the following assumption.

\begin{assumption}[Linear system]\label{ass:linear:sys} The following hold.
\begin{enumerate}[(i)]
    \itemsep0em 
	\item \label{ass:lin:sys:stability} The linear system \eqref{eq:LTI:system} is stable, \textit{i.e.,} $\theta\in\Theta$.
	\item \label{ass:lin:sys:exogeneous:noise} For each $\theta\in\Theta$ the disturbances $\{w_t\}_{t\in\mathbb{N}}$ are independent and identically distributed (i.i.d.) and independent of $x_0$ under $\mathbb{P}_\theta$. The marginal noise distributions are unbiased ($\mathbb{E}_{\theta}[w_t]=0$), non-degenerate ($S_w=\mathbb{E}_{\theta}[w_t w_t^\mathsf{T}]\succ 0$ is finite) and have an everywhere positive probability density function.
\end{enumerate}
\end{assumption}

Assumption~\ref{ass:linear:sys} ensures that the linear system~\eqref{eq:LTI:system} admits an invariant distribution~$\nu_\theta$~\autocite[\S~10.5.4]{ref:meyn-09}. 
This means that $x_{t}\sim \nu_\theta$ implies $x_{t+1}\sim \nu_\theta$ for any $t\in\mathbb{N}$. Moreover, as the probability density function of $w_t$ is everywhere positive, $\{x_t\}_{t\in\mathbb{N}}$ represents a uniformly ergodic Markov process, which implies that the marginal distribution of~$x_t$ under~$\mathbb{P}_\theta$ converges weakly to~$\nu_\theta$ as~$t$ tends to infinity~\autocite[Theorems~16.5.1 and~16.2.1]{ref:meyn-09}. 
Assumption~\ref{ass:linear:sys} then implies that the mean vector of~$\nu_\theta$ vanishes and that the covariance matrix~$S_\theta$ of~$\nu_\theta$ coincides with the unique solution of the discrete Lyapunov equation 
\begin{equation}\label{eq:Lyapunov}
S_\theta = \theta S_\theta \theta^\mathsf{T} + S_w,
\end{equation}
which provides for a convenient way to compute $S_{\theta}$; see, {\em e.g.}, \autocite[\S~6.10\,E]{ref:Antsaklis-06}. Recall that $S_\theta$ critically enters the discrepancy function~$I(\theta',\theta)$ defined in~\eqref{eq:rate:function:AR:MDP} and thus also the reverse $I$-projection defined in~\eqref{equ:rev:I:proj}. The following main theorem summarizes the key statistical and computational properties of the reverse $I$-projection that will be proved in the remainder of the paper. This theorem involves the function $\mathsf{dlqr}(A,B,Q,R)$, which outputs the optimal feedback gain matrix of an infinite-horizon deterministic LQR problem in discrete time. Such problems are described by two system matrices~$A$ and~$B$, a state cost matrix~$Q\succeq 0$ and an input cost matrix~$R\succ 0$ of compatible dimensions that satisfy standard stabilizability and detectability conditions~\autocite[\S~4]{ref:BertVolI_05}.

\begin{theorem}[Efficient identification with stability guarantees]\label{thm:eff:stable:identification:alternative}
Suppose that Assumption~\ref{ass:linear:sys} holds, { that the noise is light-tailed as well as stationary} and that $\widehat{\theta}_T$ is the least squares estimator~\eqref{eq:LS:MDP}. Then, for any $\theta\in \Theta$ the reverse $I$-projection defined in \eqref{equ:rev:I:proj} displays the following properties.
\begin{enumerate}[(i)]
    \item \label{item:i:thm} \textbf{Asymptotic consistency.}
    \begin{equation*}
        \lim_{T\to\infty}\mathcal{P}(\widehat{\theta}_T)=\theta \quad \mathbb{P}_\theta\text{-a.s.}
    \end{equation*} 
     \item\label{item:ii:thm} \textbf{Finite sample guarantee.} 
     {
     There are constants $\tau\geq 0$ and $\rho\in(0,1)$ that depend only on~$\theta$ such that
    \begin{equation*}
    \mathbb{P}_{\theta}\left( \|\theta -  \mathcal{P}(\widehat \theta_T) \|_2 \leq \kappa(S_w) \frac{2\varepsilon n^\frac{1}{2}\tau}{\sqrt{1-\rho^2}} \right)\geq 1 -\beta
\end{equation*}
for all $\beta,\varepsilon\in (0,1)$ and $T\geq\kappa(S_w) \widetilde{O}(n) {\log(1/\beta)}/{\varepsilon^2}$.
}
    \item \label{item:iii:thm}\textbf{Efficient computation.} For any $\theta'\notin\Theta$ and $S_w,Q \succ 0$ there is a~$p\ge 1$, such that for all $\delta>0$ we have that
    \begin{equation*}
        \theta^{\star}_{\delta} = \theta' + \mathsf{dlqr}(\theta',I_n,Q,(2\delta S_w)^{-1})
    \end{equation*}
    is stable and satisfies $\|\mathcal{P}(\theta')-\theta^{\star}_{\delta}\|_2 \leq O(\delta^{p})$. 
\end{enumerate}
\end{theorem}
 
The asymptotic consistency \eqref{item:i:thm} formalizes the intuitive requirement that more data is preferable to less data. We emphasize that the reverse $I$-projection does not introduce unnecessary bias because $\mathcal{P}(\theta')=\theta'$ if $\theta'$ is already stable. The finite sample guarantee~\eqref{item:ii:thm} stipulates that the projected least squares estimator $\mathcal{P}(\widehat{\theta}_T)$ is guaranteed to be close to the (unknown) true stable matrix $\theta$ with high probability~$1-\beta$. 
Note that if the observed state trajectory $\{\widehat{x}_t\}_{t=0}^T$ is generated under $\mathbb{P}_\theta$, then the inverse matrix appearing in~\eqref{eq:LS:MDP} exists $\mathbb{P}_\theta$-almost surely for any sample size $T\ge n$ thanks to Assumption~\ref{ass:linear:sys}\,\eqref{ass:lin:sys:exogeneous:noise}. The efficient computability property~\eqref{item:iii:thm}, finally, shows that computing the reverse $I$-projection to within high accuracy is no harder than solving a standard LQR problem. The function~$\mathsf{dlqr}$ is readily available as a standard routine in software packages such as MATLAB or Julia. We also emphasize that setting $Q=I_n$ works well in practice, that is, no tuning is required to compute~$\mathcal P(\theta')$. { However, tuning~$Q$ can nevertheless improve the conditioning of the optimization problem and speed up the computation of~$\mathcal P(\theta')$. Guidelines on choosing~$Q$ and the results of extensive numerical experiments are reported in~\cite{ref:revI2022}.}  Recall from Section~\ref{sec:contributions} that the reverse $I$-projection exhibits optimism in the face of uncertainty, a decision-making paradigm that is used with great success in various reinforcement learning applications~\autocite{ref:lattimore-2020}. In general, however, optimism in the face of uncertainty leads to computational intractability~\autocite{ref:Campi_98}. Thus, the tractability result of  Theorem~\ref{thm:eff:stable:identification:alternative}~\eqref{item:iii:thm} is a perhaps unexpected exception to this rule; see Proposition~\ref{prop:numerical:comp} below for further details.
In the remainder we will prove Theorem~\ref{thm:eff:stable:identification:alternative}. 
The proofs of auxiliary results  are relegated to Section~\ref{sec:appendix} in the appendix. 

\section{Reverse $I$-projection}
\label{sec:rev:I}

We now demonstrate that the discrepancy function~\eqref{eq:rate:function:AR:MDP} underlying the reverse $I$-projection has a natural statistical interpretation, which is crucial for the proof of Theorem~\ref{thm:eff:stable:identification:alternative}.  

\subsection{Moderate Deviations Theory}
\label{sec:moderate:deviations:theory}
We leverage recent results from moderate deviations theory to show that the discrepancy function~\eqref{eq:rate:function:AR:MDP} is intimately related to the least squares estimator~\eqref{eq:LS:MDP}. 
To this end, we first introduce the basic notions of a {\em rate function} and a {\em moderate deviation principle}. For a comprehensive introduction to moderate deviations theory we refer to~\autocite{hollander2008largedeviationis, dembo2009large}. 

\begin{definition}[Rate function]
	\label{def:rate_function:original}
	An extended real-valued function $I:\Theta'\times \Theta\rightarrow [0,\infty]$ is called a rate function if it is lower semi-continuous in its first argument.
\end{definition}

\begin{definition}[Moderate deviation principle]
	\label{def:MDP:general}
	A sequence of estimators $\{\widehat\theta_T\}_{T\in\mathbb{N}}$ is said to satisfy a moderate deviation principle with rate function $I$ if for every sequence $\{a_T\}_{T\in\mathbb{N}}$ of real numbers with $1\ll a_T\ll T$, for every Borel set $\mathcal{D}\subset \Theta'$ and for every $\theta\in\Theta$ { all of the following inequalities hold.}
	\begin{subequations}
		\label{eq:ldp_exponential_rates}
		\begin{align}
		\label{eq:ldp_exponential_rates_lb}
		-\inf_{\theta' \in \mathsf{int}{\mathcal{D}}} \, I(\theta', \theta) \leq& \liminf_{T\to \infty}~\frac{1}{a_T} \log \mathbb{P}_{\theta}\left( \widehat \theta_T \in \mathcal{D} \right) \\
		\leq& \limsup_{T\to \infty}~\frac{1}{a_T} \log \mathbb{P}_\theta\left( \widehat \theta_T \in \mathcal{D} \right)\\ 
		 \leq &-\inf_{\theta' \in \mathsf{cl} {\mathcal{D}}} \,  I(\theta', \theta)
		\label{eq:ldp_exponential_rates_ub}
		\end{align}
	\end{subequations}
\end{definition}
 If the rate function $I(\theta',\theta)$ is continuous in $\theta'$ and the interior of~$\mathcal{D}$ is dense in $\mathcal{D}$, then the infima in~\eqref{eq:ldp_exponential_rates_lb} and~\eqref{eq:ldp_exponential_rates_ub} coincide, which implies that all inequalities in~\eqref{eq:ldp_exponential_rates} collapse to equalities. In this case,~\eqref{eq:ldp_exponential_rates} can be paraphrased as~$\mathbb{P}_\theta(\widehat \theta_T\in\mathcal{D})= e^{-r a_T+o(a_T)}$, where~$r=\inf_{\theta' \in \mathcal{D}} I(\theta', \theta)$ represents the $I$-distance between the system matrix~$\theta$ and the set~$\mathcal{D}$ of estimator realizations. 
 Thus, $r$ represents the decay rate of the probability~$\mathbb{P}_\theta(\widehat \theta_T\in\mathcal{D})$, while~$\{a_T\}_{T\in\mathbb{N}}$ can be viewed as the speed of convergence. The condition~$1\ll a_T\ll T$ is satisfied, for example, if~$a_T=\sqrt{T}$, $T\in\mathbb{N}$. However, many other choices are possible. It is perhaps surprising that if a sequence of estimators satisfies a moderate deviations principle, then the choice of the speed~$\{a_T\}_{T\in\mathbb{N}}$ has no impact on the decay rate~$r$ but may only influence the coefficients of the higher-order terms hidden in $o(a_T)$. We also remark that if the inequalities in~\eqref{eq:ldp_exponential_rates} hold for $a_T=T$, $T\in\mathbb{N}$ (in which case the speed of convergence violates the condition $1\ll a_T\ll T$), then $\{\widehat\theta_T\}_{T\in\mathbb{N}}$ is said to satisfy a {\em large} deviation principle \autocite{dembo2009large}. It is also customary to talk about a moderate deviation principle as being a large deviation principle with reduced speed $\{a_T\}_{T\in \mathbb{N}}$ such that $1\ll a_T \ll T$.

We now show that the transformed least squares estimators
\begin{equation} \label{eq:transformed:LSE}
    \widehat \vartheta_T 
    =\sqrt{T/a_T}(\widehat{\theta}_T-\theta)+\theta
\end{equation}
satisfy a moderate deviation principle, where the discrepancy function~\eqref{eq:rate:function:AR:MDP} plays the role of the rate function. This result relies on another standard regularity condition.

\begin{assumption}[Light-tailed noise and stationarity]\label{ass:model:MDP:regularity} The following hold for every $\theta\in\Theta$.
\begin{enumerate}[(i)]
	\item \label{ass:control:regularity:i} The disturbances $\{w_t\}_{t\in\mathbb{N}}$ are
	light-tailed, \textit{i.e.}, there exists $\alpha >0$ with $\mathbb{E}_{\theta}[e^{\alpha \|w_t\|^2}]< \infty$ for all $t\in\mathbb{N}$.
	\item \label{ass:control:regularity:ii} The initial distribution $\nu$ coincides with the invariant distribution $\nu_\theta$ of the linear system \eqref{eq:LTI:system}.
\end{enumerate}
\end{assumption}

Assumption~\ref{ass:model:MDP:regularity}~\eqref{ass:control:regularity:i} essentially requires the tails of the noise to have no heavier tails than a normal distribution, while Assumption~\ref{ass:model:MDP:regularity}~\eqref{ass:control:regularity:ii} stipulates that the linear system is in the stationary regime already at time~$t=0$.

\begin{proposition}[Moderate deviation principle]
\label{prop:MDP:closed:loop}
	If Assumptions~\ref{ass:linear:sys} and~\ref{ass:model:MDP:regularity} hold, $\{\widehat \theta_T\}_{T\in\mathbb{N}}$ denote the least squares estimators defined in~\eqref{eq:LS:MDP} and  $\{a_T\}_{T\in\mathbb{N}}$ is a real sequence with $1\ll a_T\ll T$, then the transformed least squares estimators
	$\{\sqrt{T / a_T}(\widehat \theta_T - \theta) + \theta\}_{T\in\mathbb{N}}$ satisfy a moderate deviation principle with rate function~\eqref{eq:rate:function:AR:MDP}.
\end{proposition}

Unlike the standard least squares estimators~\eqref{eq:LS:MDP}, the transformed estimators of Proposition~\ref{prop:MDP:closed:loop} depend on the unknown parameter~$\theta$. However, as we will explain below, they are useful for theoretical considerations. Proposition~\ref{prop:MDP:closed:loop} can be viewed as a corollary of~\autocite[Theorem~2.1]{ref:Yu-09}, which uses ideas from \autocite{ref:worms1999} to show that the transformed least squares estimators satisfy a moderate deviation principle with a rate function that is defined implicitly in variational form. Proposition~\ref{prop:MDP:closed:loop} shows that this rate function admits the explicit representation~\eqref{eq:rate:function:AR:MDP} and allows for showing~\eqref{equ:pdf}. It also relaxes the restrictive  condition~$\|\theta\|_2<1$ from~\autocite[Proposition~2.2]{ref:Yu-09} to~$\rho(\theta)<1$.

By identifying the discrepancy function~\eqref{eq:rate:function:AR:MDP} with the rate function of a moderate deviation principle, Proposition~\ref{prop:MDP:closed:loop} justifies our terminology, whereby~$\mathcal P(\theta')$ is called the reverse $I$-projection of~$\theta'$. Indeed, \textcite{ref:csiszar-03} use this term to denote any projection with respect to an information divergence~$I(\theta',\theta)$. Note that swapping the arguments $\theta'$ and $\theta$ of the (asymmetric) function~$I(\theta',\theta)$ would give rise to an ordinary $I$-\textit{projection}~\autocite{ref:csiszar-84}. Proposition~\ref{prop:MDP:closed:loop} also suggests that the reverse $I$-projection is intimately related to maximum likelihood estimation, as already alluded to in the introduction. Indeed, for i.i.d.\ training data it is well-known that every maximum likelihood estimator can be regarded as a reverse $I$-projection with respect to the rate function of some large deviation principle \autocite[Lemma 3.1]{ref:book:Csiszar-04}.

The power of Proposition~\ref{prop:MDP:closed:loop} lies in its generality. Indeed, a moderate deviation principle provides tight bounds on the probability of \textit{any} Borel set of estimator realizations. A simple direct application of the moderate deviation principle established in Proposition~\ref{prop:MDP:closed:loop} is described below.

\begin{example}[System identification]
\upshape{
Consider a scalar system with $S_w=1$ that satisfies Assumptions~\ref{ass:linear:sys} and~\ref{ass:model:MDP:regularity}. In this case ~$\Theta=(-1,1)$ with the rate function~\eqref{eq:rate:function:AR:MDP} reducing to~$I(\theta',\theta)=\frac{1}{2}(\theta'-\theta)^2/(1-\theta^2)$. Using the least squares estimators~\eqref{eq:LS:MDP} to identify~$\theta$, Proposition~\ref{prop:MDP:closed:loop} reveals that
\begin{align*}
    \mathbb{P}_{\theta}(|\widehat{\theta}_T-\theta|\! &> \varepsilon \sqrt{a_T/T})\\
    &= \mathbb{P}_{\theta}(\theta+\sqrt{T/a_T}(\widehat{\theta}_T-\theta)\in\mathcal D) \\
    & = \textstyle \exp \left(- \inf_{\theta'\in\mathcal D} I(\theta',\theta)\cdot a_T +o(a_T)\right) \\
    &= \textstyle \exp\left( -\frac{1}{2}\varepsilon^2\,a_T/(1-\theta^2) +o(a_T)\right)
\end{align*}
for any $\varepsilon>0$ and $T\in\mathbb{N}$, where $\mathcal D=\{\theta'\in\mathbb{R}: |\theta'-\theta|>\varepsilon\}$. This result confirms the counterintuitive insight of~\textcite{ref:Sim_18} whereby stable systems with $|\theta|\approx 1$ are easier to identify than systems with $|\theta|\approx 0$.
}
\end{example}

{
\begin{remark}[Invariance under noise scaling]
The rate function~$I$ is invariant under any strictly positive scaling of the noise covariance matrix~$S_w$. This implies that if the state is one-dimensional or~$S_w$ is known to be isotropic, then~$I$ is independent of~$S_w$. For a proof see~\cite[Example~II.9]{ref:topoID2021}. 
\end{remark}
}

The moderate deviation principle established in Proposition~\ref{prop:MDP:closed:loop} also enables us to find statistically optimal data-driven decisions for stochastic optimization problems, where the underlying probability measure is only indirectly observable through finitely many training samples. Indeed, \textcite{ref:vanParys:fromdata-17} and \textcite{ref:Sutter-19} show that such optimal decisions can be found by solving data-driven distributionally robust optimization problems.

Next, we establish several structural properties of the rate function~\eqref{eq:rate:function:AR:MDP}.

\begin{proposition}[Properties of~$I(\theta',\theta)$]
	\label{prop:I_inf}
	The rate function~$I(\theta',\theta)$ defined in~\eqref{eq:rate:function:AR:MDP} has the following properties.
	\begin{enumerate}[(i)]
	    \item\label{prop:I_inf:cont} $I(\theta',\theta)$ is analytic in~$(\theta' ,\theta)\in\Theta'\times \Theta$. 
	    \item\label{prop:I_inf:compact} If~$\theta'\in\Theta$, then the sublevel set~$\{\theta\in\Theta:I(\theta', \theta)\leq r\}$ is compact for every~$r\ge 0$. 
	    \item If~$\theta'\in\Theta$, then $I(\theta',\theta)$ tends to infinity as~$\theta$ approaches the boundary of~$\Theta$. 
	\end{enumerate}
\end{proposition}
Proposition~\ref{prop:I_inf}~\eqref{prop:I_inf:compact} guarantees that the minimum in~\eqref{equ:rev:I:proj} is indeed attained and that the reverse $I$-projection is well-defined. 
To close this section, we present a useful relation between the rate function $I$ and the operator norm. 
\begin{lemma}[Pinsker-type inequality]
\label{lem:rate:to:norm}
For any~$\theta'\in\Theta'$ and $\theta\in \Theta$ we have $\|\theta'-\theta\|_2^2 \leq {2\kappa(S_w)\cdot I(\theta',\theta)}$.
\end{lemma}
Lemma~\ref{lem:rate:to:norm} provides a direct link between the nearest stable matrix problem~\eqref{equ:standard:proj} and the reverse $I$-projection~\eqref{equ:rev:I:proj} as
\begin{equation*}
    \inf_{{\theta}\in \Theta}\|\theta'-{\theta}\|_2^2\leq 2\kappa(S_w) \cdot I(\theta',\mathcal{P}(\theta')).
\end{equation*}

\subsection{Statistics of the reverse $I$-projection} \label{ssec:stability:guarantees}

In the following we apply the reverse $I$-projection to the least squares estimator~$\widehat \theta_T$ and {\em not} to the transformed least squares estimator $\widehat{\vartheta}_T$ defined in~\eqref{eq:transformed:LSE} (which is anyway unaccessible because~$\theta$ is unknown), even though Proposition~\ref{prop:MDP:closed:loop} relates $I$ to $\widehat{\vartheta}_T$. However, an elementary calculation shows that $I(\widehat \theta_T,\theta)= (a_T/T)I(\widehat{\vartheta}_T,\theta)$, and thus $I(\widehat \theta_T,\theta)$ inherits any statistical interpretations from $I(\widehat{\vartheta}_T,\theta)$. 

{We first show that $\mathcal P(\widehat \theta_T)$ is asymptotically consistent.}
\begin{proposition}[Asymptotic consistency]
\label{prop:consistency}
Suppose that Assumption~\ref{ass:linear:sys} holds and that $\widehat{\theta}_T$ is the least squares estimator. Then, for any $\theta\in \Theta$ the reverse $I$-projection $\mathcal{P}(\widehat{\theta}_T)$ of $\widehat{\theta}_T$ satisfies $\lim_{T\to\infty}\mathcal{P}(\widehat{\theta}_T)=\theta \quad \mathbb{P}_\theta\text{-a.s.}$.
\end{proposition}

Next, we can use the results of Section~\ref{sec:moderate:deviations:theory} to establish probabilistic bounds on the operator norm distance between the projected least squares estimator
$\mathcal{P}(\widehat \theta_T)$ and the unknown true system matrix~$\theta$ with respect to the data-generating probability measure~$\mathbb P_\theta$. 
Specifically, the following lemma provides {two {\em implicit}} finite-sample bounds { involving random error estimates. These bounds} 
are both structurally identical to existing (and in some cases statistically optimal) finite-sample bounds for~$\widehat{\theta}_T$; see, {\em e.g.}, \autocite[\S~6]{ref:sarkar19a}. { In Proposition~\ref{prop:finite:sample} below, these implicit bounds will be used to establish {\em explicit} finite sample bounds involving deterministic error estimates.}

{
\begin{lemma}[Implicit finite sample bounds]
\label{lem:learn:stab}
Suppose that Assumptions~\ref{ass:linear:sys} and~\ref{ass:model:MDP:regularity} hold and that $\widehat{\theta}_T$ and $\mathcal P(\widehat{\theta}_T)$ represent the least squares estimator and its reverse $I$-projection, respectively. Setting~$\widehat{\varepsilon}_T =  [2\kappa(S_w)\cdot I(\widehat\theta_T,\mathcal{P}(\widehat \theta_T))]^{1/2}$, we then have $\|\widehat{\theta}_T-\mathcal{P}(\widehat{\theta}_T)\|_2 \leq \widehat{\varepsilon}_T$ $\mathbb P_\theta$-almost surely. In addition, the following finite sample bounds hold for all $\beta,\varepsilon\in (0,1)$.
\begin{subequations}
\label{equ:lem:error:bounds}
\begin{enumerate}[(i)]
    \item We have
    \begin{equation}
\label{equ:error:finite:sample}
    \mathbb{P}_{\theta}\left( \|\theta -  \mathcal{P}(\widehat \theta_T) \|_2 \leq \varepsilon+{\widehat{\varepsilon}_T}  \right)\geq 1 -\beta
\end{equation}
for all $T\in\mathbb N$ with $T \geq  \kappa(S_w)\widetilde{O}(n)\log(1/\beta)/\varepsilon^2$.
\item\label{stat:asymp} If $\{a_T\}_{T\in\mathbb N}$ is a real sequence satisfying $1\ll a_T\ll T$, then we have
\begin{equation}
\label{equ:error:asymptotic}
    \mathbb{P}_{\theta}\left( \|\theta -  \mathcal{P}(\widehat \theta_T) \|_2 \leq \varepsilon \sqrt{{a_T}/{T}} + \widehat{\varepsilon}_T \right)\geq 1 -\beta
\end{equation}
for all $T\in \mathbb{N}$ with ${a_T}\geq 2 \kappa(S_w) (\log(1/\beta)+o(a_T))/\varepsilon^2$.
\end{enumerate}
\end{subequations}
\end{lemma}
}

{ Note that the finite sample bound~\eqref{equ:error:finite:sample}, which leverages sophisticated results from \autocite[\S~6]{ref:sarkar19a}, and the bound~\eqref{equ:error:asymptotic}, which follows almost immediately from the moderate deviations principle of Section~\ref{sec:moderate:deviations:theory}, are qualitatively similar. They both hold for all $T$ that exceed a critical sample size depending on an unknown deterministic function of the order~$\widetilde{O}(n)$
or~$o(a_T)$, respectively. Both bounds also involve a random error estimate~$\widehat\varepsilon_T$. As $\widehat \theta_T$ as well as $\mathcal P(\widehat \theta_T)$ converge $\mathbb P_\theta$-almost surely to~$\theta\in\Theta$, and as $I$ is continuous in both of its arguments, it is easy to show that the random variable~$\widehat{\varepsilon}_T$ as defined in Proposition~\ref{lem:learn:stab} converges $\mathbb P_\theta$-almost surely to~$0$ as $T$ grows. Therefore, the bounds~\eqref{equ:error:finite:sample} and~\eqref{equ:error:asymptotic} improve with~$T$.
As the inequalities in~\eqref{eq:ldp_exponential_rates} are asymptotically tight, we conjecture that the bound~\eqref{equ:error:asymptotic} is statistically optimal.

In the following we will show that the implicit finite sample bounds of Lemma~\ref{lem:learn:stab} can be used to derive explicit finite sample bounds involving deterministic error estimates. To this end, we recall a more nuanced quantitative notion of stability.

\begin{definition}[$(\tau,\rho)$-stability~{\cite[Definition~1]{ref:Krauth-19}}]
\label{def:tau-rho-stability}
    We say that the system matrix $\theta \in \Theta$ is $(\tau,\rho)$-stable for some $\tau\geq 1$ and $\rho\in (0,1)$ if $\|\theta^k\|_2\leq \tau \rho^k$ for all $k\in \mathbb{N}$. 
\end{definition}
We emphasize that any stable matrix~$\theta\in \Theta$ is in fact $(\tau,\rho)$-stable for some $\tau\geq 1$ and $\rho\in (0,1)$. If $\theta$ is diagonalizable with spectral decomposition~$\theta=T\Lambda T^{-1}$, for example, then $\|\theta^k\|=\|T\Lambda^k T^{-1}\|\leq \kappa(T)\rho(\theta)^k$, which implies that $\theta$ is $(\tau,\rho)$ stable for $\tau=\kappa(T)$ and $\rho=\rho(\theta)$. If $\theta$ is not diagonalizable, a similar but more involved argument together with a change of coordinates similar to the one from proof of Proposition~\ref{prop:MDP:closed:loop} can be used to show $(\tau,\rho)$-stability.

\begin{proposition}[Explicit finite sample bounds]
\label{prop:finite:sample}
Suppose that Assumptions~\ref{ass:linear:sys} and~\ref{ass:model:MDP:regularity} hold and that $\widehat{\theta}_T$ and $\mathcal P(\widehat{\theta}_T)$ are the least squares estimator and its reverse $I$-projection, respectively. 
The following finite sample bounds hold for all $\beta,\varepsilon\in (0,1)$ and for all parameters $\tau\geq 1$ and $\rho\in (0,1)$ such that $\theta$ is $(\tau,\rho)$-stable, which are guaranteed to exist.
\begin{subequations}
\label{equ:prop:error:bounds:2}
\begin{enumerate}[(i)]
\item We have
    \begin{equation*}
    \mathbb{P}_{\theta}\left( \|\theta -  \mathcal{P}(\widehat \theta_T) \|_2 \leq \kappa(S_w) \frac{2\varepsilon n^{\frac{1}{2}}\tau}{\sqrt{1-\rho^2}} \right)\geq 1 -\beta
\end{equation*}
for all $T\in\mathbb N$ with $T \geq  \kappa(S_w)\widetilde{O}(n)\log(1/\beta)/\varepsilon^2$.
\item\label{stat:asymp:2} If $\{a_T\}_{T\in\mathbb N}$ is a real sequence satisfying $1\ll a_T\ll T$ and $T\in \mathbb{N}$, then we have
\begin{equation*}
\mathbb{P}_{\theta}\left( \|\theta -  \mathcal{P}(\widehat \theta_T) \|_2 \leq \kappa(S_w) \frac{2\varepsilon n^\frac{1}{2}\tau}{\sqrt{1-\rho^2}}\sqrt{\frac{a_T}{T}}\right)\geq 1 -\beta
\end{equation*}
for all $T\in \mathbb{N}$ with ${a_T}\geq 2 \kappa(S_w) (\log(1/\beta)+o(a_T))/\varepsilon^2$.
\end{enumerate}
\end{subequations}
\end{proposition}

The explicit finite-sample bounds of Proposition~\ref{prop:finite:sample} refine the implicit bounds of Lemma~\ref{lem:learn:stab} and notably expose the dependence of the approximation error on the stability parameters $\tau$ and $\rho$. Of course, these parameters are unknown under our standing assumption that~$\theta$ is unknown, as such we cannot adapt the projection~\eqref{equ:rev:I:proj} to incorporate $(\tau,\rho)$-stability. In contrast, the implicit finite-sample bounds of Lemma~\ref{lem:learn:stab} involve approximation errors that are random but known.
}

\subsection{Computation of the reverse $I$-projection}
\label{sec:num:I:proj}
We now address the numerical computation of $\mathcal{P}(\theta')$ as defined in~\eqref{equ:rev:I:proj} for any given estimator realization $\theta'\in \Theta'$. To this end, we fix $Q\succ 0$ and show that solving~\eqref{equ:rev:I:proj} is equivalent to finding a minimizer of the optimization problem
\begin{equation}
\label{equ:jr}
    \min_{\theta\in \Theta}\{ \mathsf{tr}(Q S_{\theta}): I(\theta',\theta) \leq r\}
\end{equation}
for the smallest radius~$r=\underline r$ that renders~\eqref{equ:jr} feasible. Note that $\underline r$ exists because the optimal value of~\eqref{equ:jr} is lower semi-continuous in~$r$. In addition, problem~\eqref{equ:jr} admits a minimizer for any $r\ge \underline r$ due to Proposition~\ref{prop:I_inf}~\eqref{prop:I_inf:compact}. The proposed procedure works because if $\theta'$ is unstable, then any $\theta$ feasible in~\eqref{equ:jr} is stable, and its $I$-distance to~$\theta'$ is at most~$r$. Setting~$r=\underline r$ thus ensures that any minimizer of~\eqref{equ:jr} is a reverse $I$-projection of $\theta'$ and that $I(\theta',\mathcal P(\theta'))=\underline r$. Moreover, the proposed procedure is computationally attractive because we will prove below that~\eqref{equ:jr} is equivalent to a standard LQR problem. We emphasize that the exact choice of~$Q$ has no effect on the validity and hardly any effect on the numerical performance of this procedure. 

\begin{proposition}[Reformulation of~\eqref{equ:rev:I:proj}]
\label{prop:proxy}
If $\theta'\in \Theta'$, $Q\succ 0$ and $\underline r$ is the smallest $r \geq 0$ for which~\eqref{equ:jr} is feasible, then any minimizer of~\eqref{equ:jr} at $r=\underline r$ is a reverse $I$-projection. 
\end{proposition}

Note that if $r\ge \overline r=I(\theta',0)$, then problem~\eqref{equ:jr} has the trivial solution $\theta=0$, and its optimal value reduces to $\mathsf{tr}(QS_w)$.\footnote{One readily verifies that $\overline r=\frac{1}{2}\|S_w^{-1/2}\theta'S_w^{1/2}\|_F^2$, where $\|\cdot\|_{F}$ stands for the Frobenius norm.} In this case, the rate constraint is not binding at optimality. If $r<\overline r$, on the other hand, then problem~\eqref{equ:jr} is infeasible for $r<\underline r$ and admits a quasi-closed form solution for $r> \underline r$ as explained in the following proposition.

\begin{proposition}[Optimal solution of~\eqref{equ:jr}]
\label{prop:opt:sol}
	\label{prop:jr:stable}
	Suppose that Assumption~\ref{ass:linear:sys} holds. Then, for every $\theta'\notin\Theta$ there exists an analytic function $\varphi:(\underline r, \overline r)\rightarrow(0,\infty)$ that is increasing and bijective such that the following hold for all~$r\in(\underline r,\overline r)$.

	\begin{enumerate}[(i)]
	    \item For any $\delta\in(0,\infty)$ the matrix $P_{\delta}\in \mathcal{S}^n$ is the unique positive definite solution of the Riccati equation
	\begin{equation}
	\label{equ:P_ARE:best}
	\begin{aligned}
	P_{{\delta}} =&\; Q + {\theta'}^\mathsf{T} P_{{\delta}} \left(I_n+2\delta S_w P_{{\delta}}\right) ^{-1}  \theta'.
	\end{aligned} 
	\end{equation}
	\item The matrix $\theta_{\delta}^{\star}=(I_n+2\delta S_w P_{\delta})^{-1}\theta'$ is the unique solution of problem~\eqref{equ:jr} at $r=\varphi^{-1}(\delta)$, and the rate constraint is binding at optimality, i.e., $I(\theta',\theta^\star_\delta)=r$.
	\end{enumerate}

\end{proposition}

We have seen that evaluating~$\mathcal P(\theta')$ is equivalent to solving~\eqref{equ:jr} at $r=\underline r$. Unfortunately, $\underline r$ is unknown, and Proposition~\ref{prop:jr:stable} only characterizes solutions of~\eqref{equ:jr} for $r>\underline r$. However, by the properties of~$\varphi$ established in Proposition~\ref{prop:jr:stable}, we also have $\lim_{r\downarrow \underline r}\varphi(r)=0$, which is equivalent to $\lim_{\delta\downarrow 0}\varphi^{-1}(\delta)= \underline r$. A standard continuity argument therefore implies that $\lim_{\delta\downarrow 0} \theta^\star_\delta$ solves~\eqref{equ:jr} at $r=\underline r$. In practice, we may simply set $\delta$ to a small positive number and compute~$\theta^\star_\delta$ by solving~\eqref{equ:P_ARE:best} to find a high-accuracy approximation for the reverse $I$-projection~$\mathcal P(\theta')$.

\begin{proposition}[Computing the reverse $I$-projection]\label{prop:numerical:comp}
If Assumption~\ref{ass:linear:sys} holds, $\theta'\notin \Theta$ and $Q\succ 0$, then there exists a $p\ge 1$ such that for all $\delta>0$ the matrix $\theta^{\star}_{\delta}$ from Proposition~\ref{prop:jr:stable}\,(ii) is stable and satisfies
\begin{equation}
\label{equ:poly:error}
    \|\mathcal{P}(\theta')-\theta^{\star}_{\delta}\|_2 = O(\delta^{p} ).
\end{equation}
In addition, $\theta^{\star}_{\delta}$ can be computed as
\begin{equation*}
    \theta^{\star}_{\delta} = \theta' + \mathsf{dlqr}(\theta',I_n,Q,(2\delta S_w)^{-1}),
\end{equation*}
where the standard LQR routine\footnote{See \url{https://juliacontrol.github.io/ControlSystems.jl/latest/examples/example/\#LQR-design} for example.
} $\mathsf{dlqr}(\cdot)$ has time and memory complexity of the order $O(n^3)$ and $O(n^2)$, respectively. 
\end{proposition}

\begin{corollary}[$\mathcal P(\theta')$ and $\theta^\star_\delta$ preserve the structure of $\theta'$]
\label{cor:struc}
For any $\theta'\in \Theta'$ there exist invertible matrices $\Lambda,\Lambda_{\delta}\in \mathbb{R}^{n\times n}$ such that $\mathcal{P}(\theta')=\Lambda^{-1}\theta'$ and $\theta^{\star}_{\delta}=\Lambda^{-1}_{\delta}\theta'$.    
\end{corollary}

Corollary~\ref{cor:struc} implies, among other things, that the reverse $I$-projection preserves the kernel of $\theta'$, see~\autocite{ref:Topo2020} for more information. {Combining Theorem~\ref{thm:eff:stable:identification:alternative} and Corollary~\ref{cor:struc} also facilitates topological linear system identification~\cite{ref:topoID2021}.} \textcite{ref:guglielmi2018closest} have shown that if $\theta'=\alpha \mathbbold{1}_{n\times n}$ with $\alpha\in [\frac{1}{n},\frac{2}{n}]$, then the solution of the closest stable matrix problem~\eqref{equ:standard:proj} with respect to the Frobenius norm is $\Pi_\Theta(\theta')=\frac{1}{n}\mathbbold{1}_{n\times n}$, which lies on the boundary of~$\Theta$. A simple calculation further shows that $\mathcal{P}(\theta')=\frac{1}{2n}\mathbbold{1}_{n\times n}$, which lies in the interior of $\Theta$ and has the same structure as~$\theta'$ and~$\Pi_\Theta(\theta')$, thus exemplifying Corollary~\ref{cor:struc}. For $\alpha>\frac{2}{n}$ problem~\eqref{equ:standard:proj} appears to have many local minima, while the reverse $I$-projection remains unique as well as structurally similar to~$\theta'$.

\begin{figure*}[h!]
    \centering
    \begin{subfigure}[b]{0.22\textwidth}
        \includegraphics[width=\textwidth]{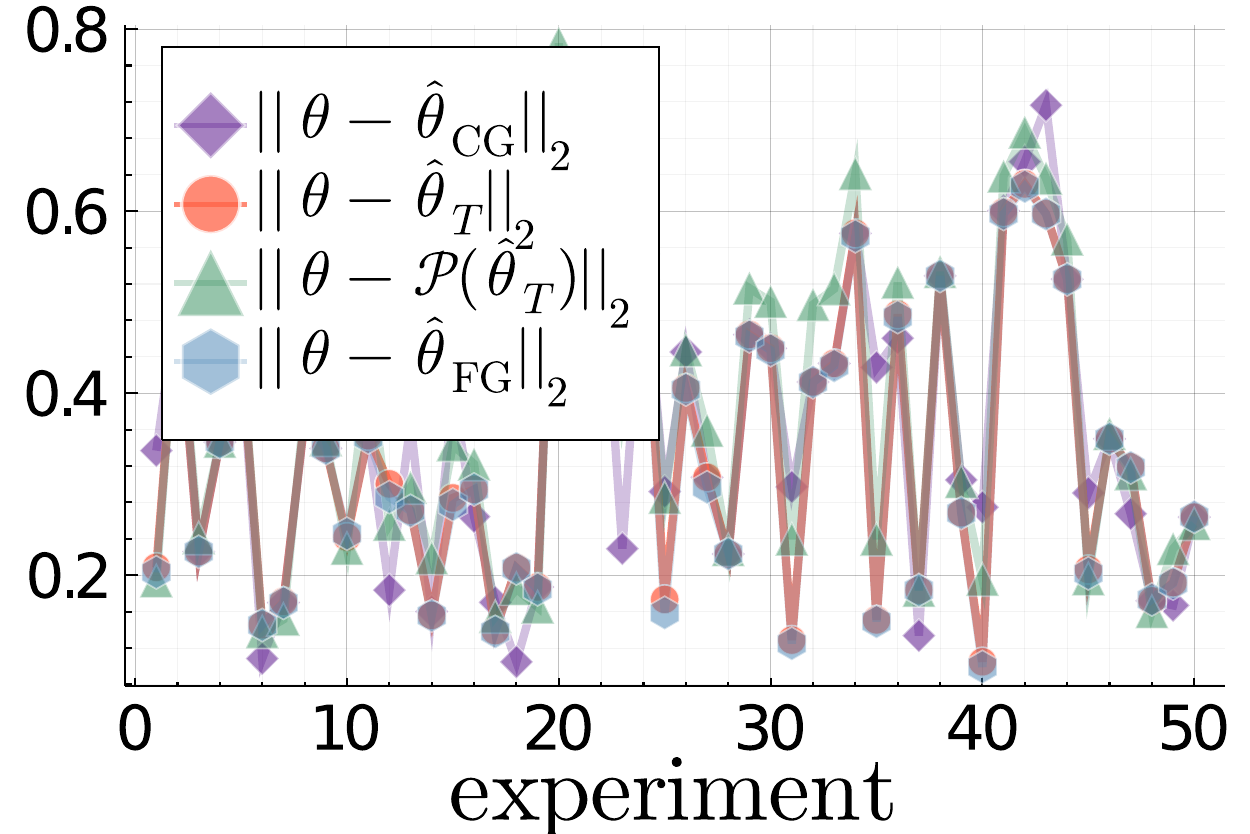}
        \caption{Approximation in operator norm for $n=3$.}
        \label{fig:learn:norm}
    \end{subfigure}\quad
    \begin{subfigure}[b]{0.22\textwidth}
        \includegraphics[width=\textwidth]{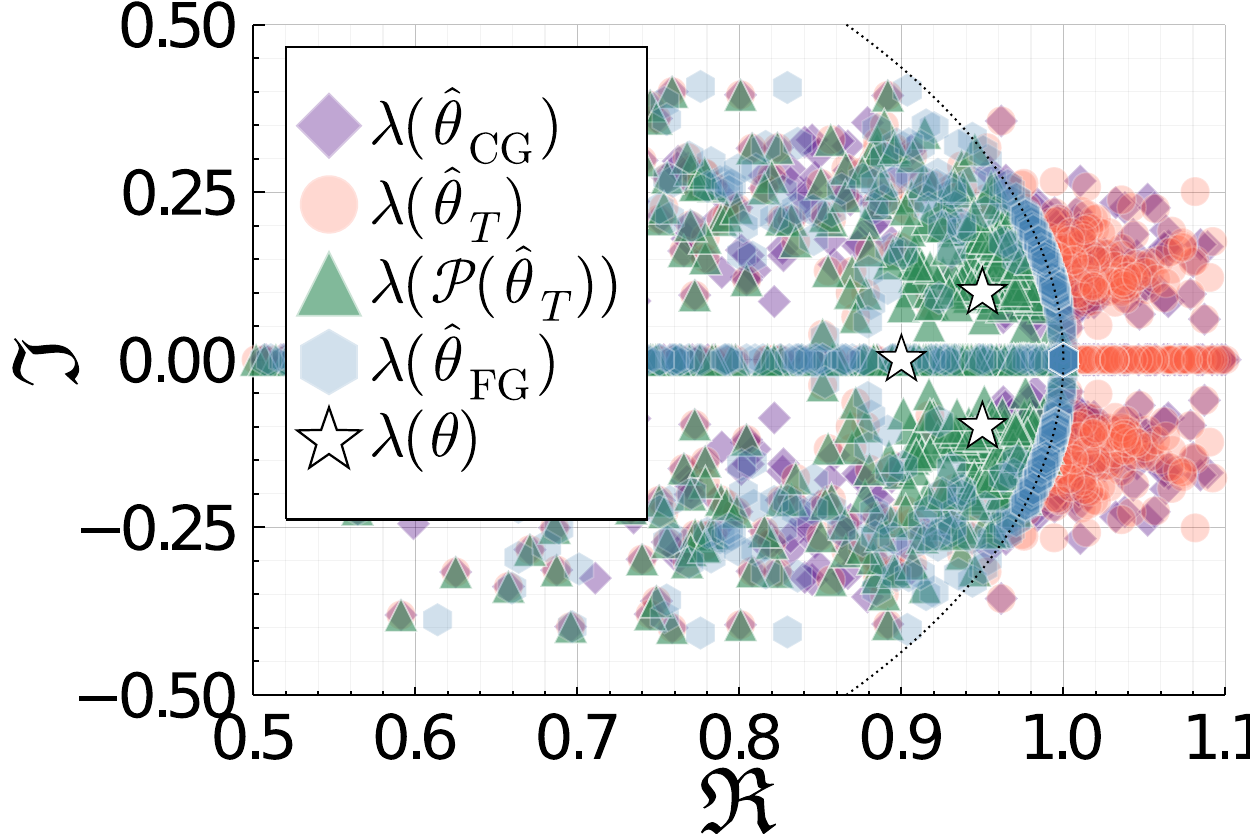}
        \caption{Eigenvalue spectra for $n=3$.}
        \label{fig:learn:eig}
    \end{subfigure}\quad 
    \begin{subfigure}[b]{0.22\textwidth}
        \includegraphics[width=\textwidth]{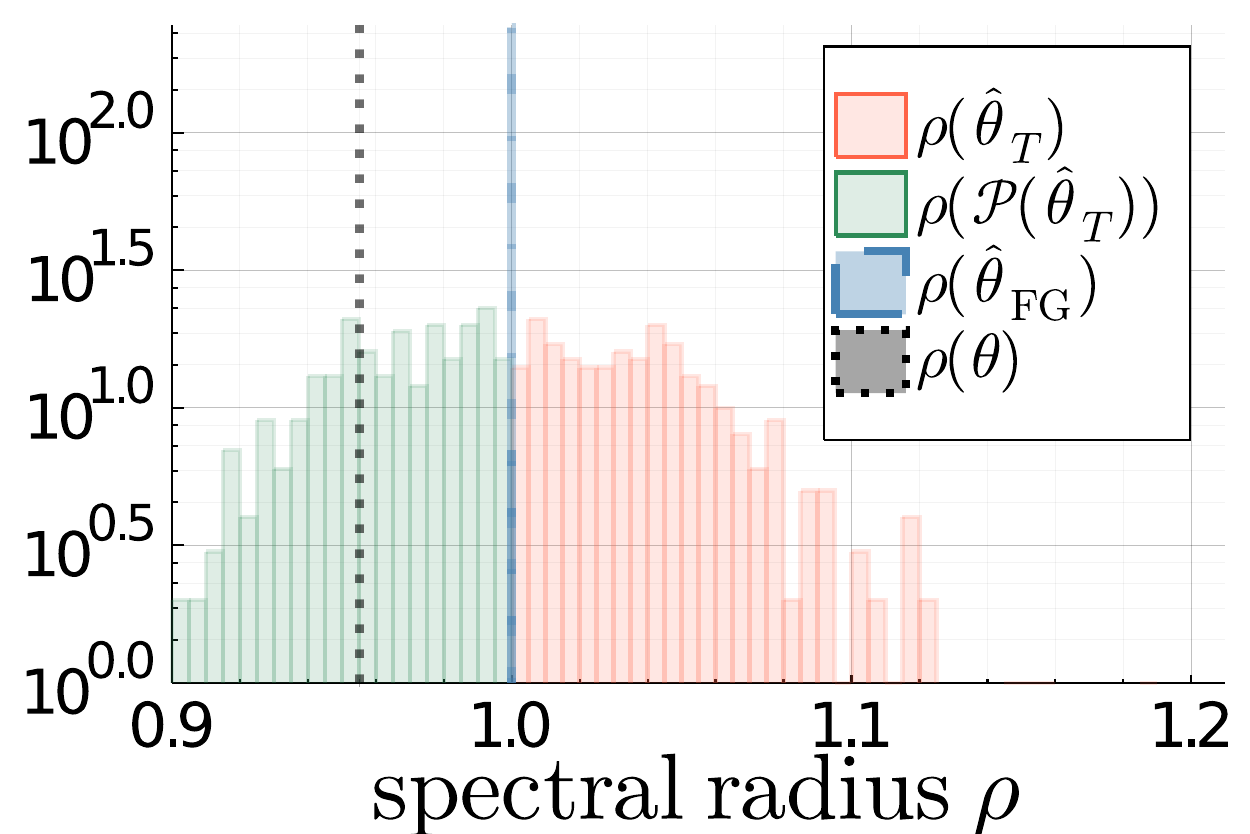}
        \caption{Historgrams of the spectral radii for $n=3$.}
        \label{fig:m1}
    \end{subfigure}\quad
    \begin{subfigure}[b]{0.22\textwidth}
        \includegraphics[width=\textwidth]{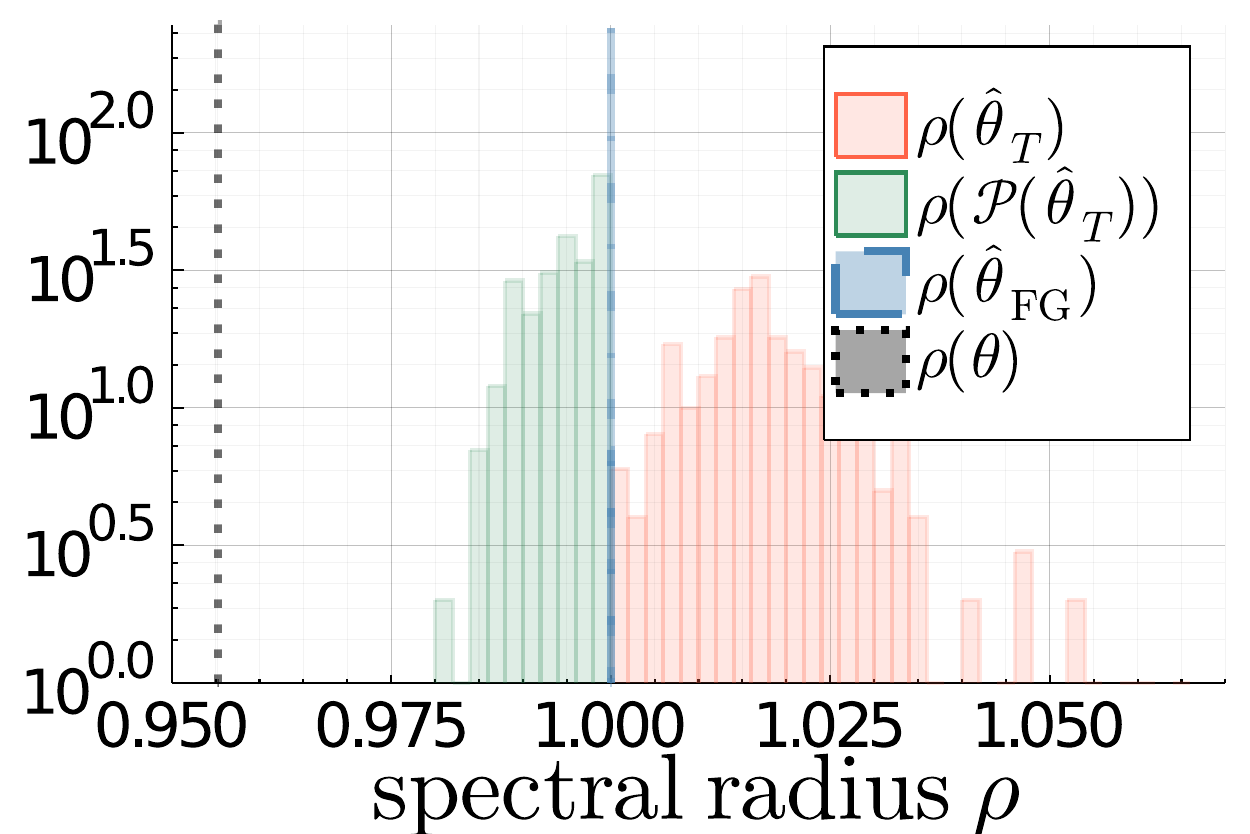}
        \caption{Historgrams of the spectral radii for $n=27$.}
        \label{fig:m9}
    \end{subfigure}
    \caption[]{Comparison of the least squares estimator~$\widehat\theta_T$ and its reverse $I$-projection $\mathcal P(\widehat\theta_T)$ against the CG and FG estimators~$\widehat \theta_{\rm CG}$ and~$\widehat \theta_{\rm FG}$, respectively, based on 250 independent simulation runs with~$\widehat\theta_T\notin \Theta$.
    To aid visibility, Figure~\ref{fig:learn:norm} only shows the first 50 experiments.}
    \label{fig:learn:stab}
\end{figure*}

\begin{figure*}[h!]
    \centering
    \begin{subfigure}[b]{0.22\textwidth}
        \includegraphics[width=\textwidth]{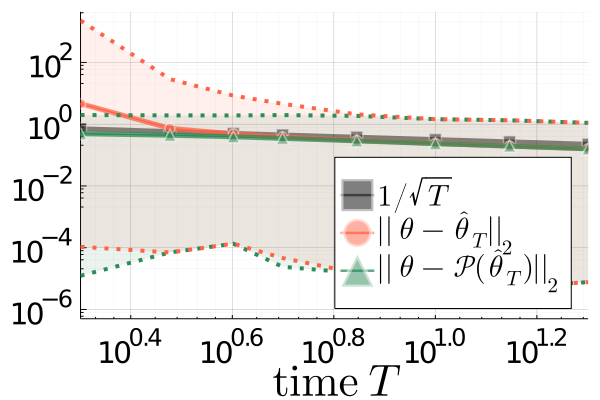}
        \caption{Convergence in operator norm for $n=1$.}
        \label{fig:learn:norm:n1}
    \end{subfigure}\quad
    \begin{subfigure}[b]{0.22\textwidth}
        \includegraphics[width=\textwidth]{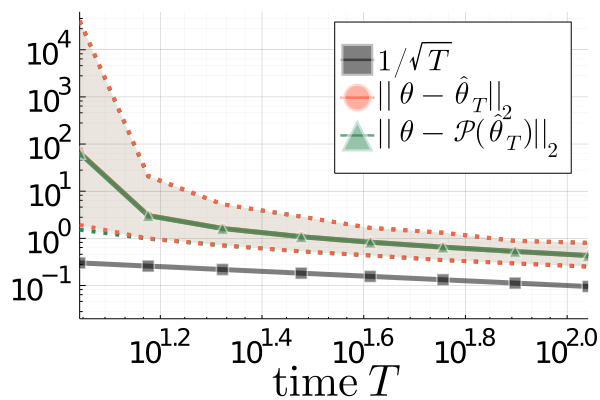}
        \caption{Convergence in operator norm for $n=10$.}
        \label{fig:learn:norm:n10}
    \end{subfigure}\quad 
    \begin{subfigure}[b]{0.22\textwidth}
        \includegraphics[width=\textwidth]{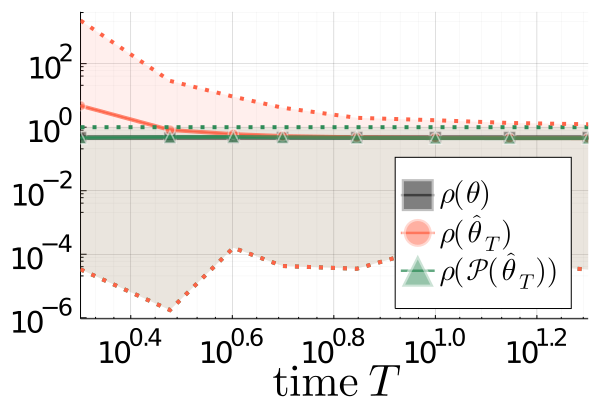}
        \caption{Convergence of spectral radii for $n=1$.}
        \label{fig:learn:rad:n1}
    \end{subfigure}\quad
    \begin{subfigure}[b]{0.22\textwidth}
        \includegraphics[width=\textwidth]{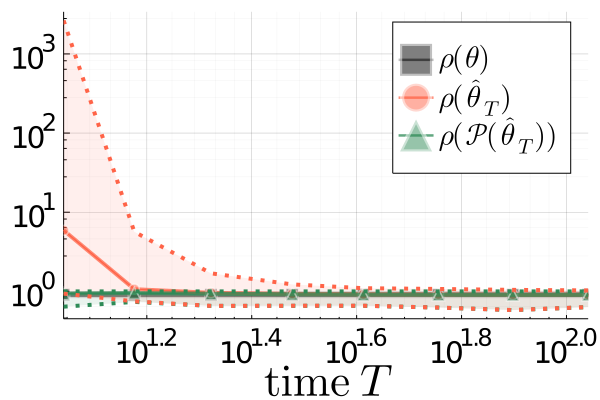}
        \caption{Convergence of spectral radii for $n=10$.}
        \label{fig:learn:rad:n10}
    \end{subfigure}
    \caption[]{Convergence behavior of~$\widehat\theta_T$ and~$\mathcal P(\widehat\theta_T)$. Solid lines represent averages and shaded areas represent ranges across $10^4$ simulations (corresponding to 100 randomly generated system matrices~$\theta$ and 100 randomly generated state trajectories per system matrix).}
    \label{fig:learn:stab:random}
\end{figure*}

Now we have all the tools in place to prove Theorem~\ref{thm:eff:stable:identification:alternative}. 
\begin{proof}[Proof of Theorem~\ref{thm:eff:stable:identification:alternative}] 
The three assertions follow directly from Propositions~\ref{prop:consistency}, \ref{prop:finite:sample} and~\ref{prop:numerical:comp}, respectively.
\end{proof}

\section{Numerical simulations}
\label{sec:num}
The two subsequent examples showcase the statistical and computational properties of the reverse $I$-projection.

\begin{example}[Spectral approximation quality]
\label{ex:spectral:approx}
\upshape{
Select $m\in \{1,9,64\}$, and set $Q=I_{3m}$, $w_t\sim\mathcal N(0,I_{3m})$ and
\begin{equation*}
    \theta = (Y\otimes I_m),\quad \text{where} \quad Y = \begin{bmatrix} 0.95 & 0.1 & 1\\
    -0.1 & 0.95 & 0\\
    0 & 0 & 0.9 \end{bmatrix}. 
\end{equation*}
The eigenvalues of $Y$ are $0.9$ and $0.95\pm i\cdot 0.1$, where $i$ is the imaginary unit, and thus $\theta$ is almost unstable. Set now $T=25 \sqrt{m}$, and generate 250 independent state trajectories for which $\widehat{\theta}_T\notin \Theta$. { This is achieved by sampling an indefinite number of state trajectories and disregarding all those for which $\widehat{\theta}_T\in \Theta$. Sampling continues until 250 state trajectories with $\widehat{\theta}_T\notin \Theta$ have been found}. Next, compute $\mathcal P(\widehat\theta_T)$ approximately as described in Proposition~\ref{prop:numerical:comp} for $\delta=10^{-9}$. In addition, compute~$\Pi_\Theta(\widehat \theta_T)$ with respect to the Frobenius norm by using the approximate constraint generation (CG) method of~\textcite{ref:Boots:Nips} and the exact fast gradient (FG) method of~\textcite{ref:Gillis-19}.\footnote{\url{sites.google.com/site/nicolasgillis/code}} Figure~\ref{fig:learn:norm} shows that for $m=1$ all methods succeed in approximating~$\theta$ reasonably closely, with the FG method having a slight edge. However, from Figure~\ref{fig:learn:eig} it becomes apparent that $\mathcal P(\widehat\theta_T)$ approximates the eigenvalue spectrum of $\theta$ best. All of its eigenvalues reside within the complex unit circle and concentrate near the true spectrum of~$\theta$, which might be explained by the structure-preserving property of the reverse $I$-projection established in Corollary~\ref{cor:struc}. In contrast, the CG method often produces unstable estimators, and the FG method generates estimators that reside on the boundary of~$\Theta$. These observations are consistent with Figure~\ref{fig:m1}, which displays the empirical distribution of the spectral radii corresponding to the different estimators. Indeed, the histogram corresponding to the reverse $I$-projection is confined to $[0,1]$ and centred around~$\rho(\theta)$. The FG method, on the other hand, is designed to generate estimators with unit spectral radius, which could, however, be undesirable in applications. {Our numerical experiments suggest that the event $\widehat{\theta}_T\notin \Theta$ becomes less likely in higher dimensions and that, if this event occurs,} then $\widehat{\theta}_T$ {concentrates near} the boundary of~$\Theta$, see Figure~\ref{fig:m9}. Thus, $\mathcal{P}(\widehat{\theta}_T)$ is more likely to have a spectral radius close to~$1$. This phenomenon is further accentuated for $m=64$, see Figure~\ref{fig:m64} in the appendix. 

As both the reverse $I$-projection and the FG method have complexity~$O(n^3)$, we compare their runtimes for $\theta' = (Y\otimes 2 I_m)\notin \Theta$ as a function of~$n=3m$, see Figure~\ref{fig:runtime} in the appendix. We observe that the reverse $I$-projection is faster for $n\lesssim 500$, while the FG method dominates for higher dimensions. We remark that the reverse $I$-projection is computed using off-the-shelf software but could be sped up by using dedicated large-scale algorithms \autocite{ref:Laub:parallel,ref:Fassbender}.
All simulations were implemented in Julia~\autocite{ref:Julia2017} and run on a 4GHz CPU with 16Gb RAM.
}
\end{example}

{
\begin{remark}[High-dimensional least squares estimators]
\label{rem:high:dim} 
It appears that the least squares estimator $\widehat{\theta}_T$ is less likely to be unstable in higher dimensions. In the context of Example~\ref{ex:spectral:approx}, the sample size~$T$ required for~$\widehat \theta_T$ to be stable with a given confidence grows indeed sublinearly with the dimension~$m$. Specifically, our experiments indicate that for $T=25 \sqrt{m}$, one needs approximately $1.75$, $1.1$ or $1$ experiments on average to generate a stable estimator for~$m=1$, $m=9$ and $m=64$, respectively. However, for $T=25 m$, one needs approximately $1.75$, $3.1$ or~$1,600$ experiments on average to generate a stable estimator for~$m=1$, $m=9$ and $m=64$. Note that these empirical frequencies may still depend on $Y$ and $S_w$.
\end{remark}
}

\begin{example}[Statistical guarantees]
\label{ex:stat:rates}
\upshape{
The second experiment is designed to validate the statistical guarantees of Proposition~\ref{prop:finite:sample}. To this end, choose $n\in\{1,10,100\}$, and sample 100 stable matrices from a standard normal distribution on~$\mathbb R^{n\times n}$ restricted to~$\Theta$. For each such matrix~$\theta$, generate 100 state trajectories of length $\overline T=10^2(n+1)$, and compute $\mathcal P(\widehat\theta_T)$ for every $T=1,\ldots,\overline T$ approximately as described in Proposition~\ref{prop:numerical:comp} for $\delta=10^{-9}$. Figures~\ref{fig:learn:norm:n1} and~\ref{fig:learn:norm:n10} visualize the convergence of the estimators $\widehat\theta_T$ and $\mathcal P(\widehat\theta_T)$ to~$\theta$ with respect to the operator norm for $n=1$ and $n=10$, respectively. Both figures are consistent with the $1/\sqrt{T}$ scaling law anticipated by Proposition~\ref{prop:finite:sample}. Although Example~\ref{ex:spectral:approx} revealed that $\rho(\mathcal{P}(\widehat{\theta}_T))$ can concentrate away from $\rho(\theta)$ in high dimensions, Figures~\ref{fig:learn:rad:n1} and~\ref{fig:learn:rad:n10} show that  $\rho(\mathcal{P}(\widehat{\theta}_T))$ converges to~$\rho(\theta)$ on average. { Figures~\ref{fig:learn:norm:n10} and~\ref{fig:learn:rad:n10} further show that the reverse $I$-projection does not need to introduce a large distortion with respect to the operator norm in order to stabilize $\widehat{\theta}_T$, \textit{i.e.}, we observe that $\|\theta-\widehat{\theta}_T\|_2\approx \|\theta-\mathcal{P}(\widehat{\theta}_T)\|_2$.} Figures~\ref{fig:learn:norm:n100} and~\ref{fig:learn:rad:n100} in the appendix extend these results to~$n=100$. 
}
\end{example}

\newpage
\section{Appendix}
\label{sec:appendix}
In this appendix we collect all proofs not contained in the main body of the paper, and we provide some auxiliary results. 

\subsection{Proofs of Section~\ref{sec:moderate:deviations:theory}}

\begin{proof}[Proof of Proposition~\ref{prop:MDP:closed:loop}]
    Fix any~$\theta\in\mathbb{R}^{n\times n}$ and assume that~$\|\theta\|_2<1$. This condition is stronger than Assumption~\ref{ass:linear:sys}\,(i) because the spectral radius $\rho(\theta)$ is bounded above by the spectral norm~$\|\theta\|_2$. Together with Assumptions~\ref{ass:linear:sys}\,(ii) and \ref{ass:model:MDP:regularity}, this condition implies via~\autocite[Proposition 2.2]{ref:Yu-09}, that the transformed least squares estimators~$\{\sqrt{T / a_T}(\widehat \theta_T - \theta)+\theta\}_{T\in\mathbb{N}}$
	satisfy a moderate deviation principle with rate function
	\begin{align} 
	\label{eq:rate:function:AR:MDP:proof}  \sup_{L\in\mathbb{R}^{n\times n}}   \left\{  \langle L, \theta^\prime- \theta\rangle - \frac{1}{2} \mathbb{E}_{\theta}\left[{\langle L, w_1 x_0^\mathsf{T} S_\theta^{-1}}\rangle^2\right]  \right\},
	\end{align}
	where the inner product of two matrices $A,B\in\mathbb{R}^{n\times n}$ is defined as $\langle A,B \rangle =\mathsf{tr}(A^{\mathsf{T}}B)$. This rate function, which is defined implicitly as the optimal value of an optimization problem, captures the speed at which the transformed least squares estimators (and indirectly also the standard least squares estimators) converge to~$\theta$.
	
	Next, we will demonstrate that~\eqref{eq:rate:function:AR:MDP:proof} is equivalent to~$I(\theta',  \theta)$ defined in~\eqref{eq:rate:function:AR:MDP}. As a preparation, we derive the analytical solution of the following unconstrained convex quadratic maximization problem over the matrix space $\mathbb{R}^{n\times n}$,
	\begin{equation} \label{LQG:proof:auxiliary:op}
    	\max_{X\in\mathbb{R}^{n\times n}}\left\{ \langle C, X \rangle - \frac{1}{2} \mathsf{tr}\left(X B_1 X^\mathsf{T} B_2\right) \right\},
	\end{equation}
	which is parameterized by $B_1, B_2\in\mathcal{S}^n_{\succ 0}$ and  $C\in\mathbb{R}^{n\times n}$. As the trace term $\mathsf{tr}(X B_1 X^\mathsf{T} B_2)$ is convex in $X$ by virtue of \autocite[Corollary~1.1]{ref:Lieb-73}, we can solve~\eqref{LQG:proof:auxiliary:op} by setting the gradient of the objective function to zero. Specifically, using \autocite[Propositions~10.7.2 \& 10.7.4]{ref:Ber-09}, we find
	\begin{align*}
	{\nabla_X} \left( \langle{C},{X}\rangle - \frac{1}{2}\mathsf{tr}(X B_1 X^\mathsf{T} B_2)  \right)= C^\mathsf{T} -  B_1 X^\mathsf{T} B_2.
	\end{align*}
	As $B_1,B_2\succ 0$, it is easy to verify that this gradient vanishes at $X^\star = B_2^{-1}CB_1^{-1}$, which implies that the optimal value of problem~\eqref{LQG:proof:auxiliary:op} amounts to $\frac{1}{2}\mathsf{tr}(B_2^{-1}CB_1^{-1}C^\mathsf{T})$. 
	
	Next, we rewrite the expectation in~\eqref{eq:rate:function:AR:MDP:proof} as
	\begin{align*}
	\mathbb{E}_{\theta}\left[{\langle{L},{w_1 x_0^\mathsf{T} S_\theta^{-1}}}\rangle^2\right] &
	= \mathbb{E}_{\theta}\left[{(w_1^\mathsf{T} L S_\theta^{-1} x_0 )^2}\right]\\
	&= \mathbb{E}_{\theta}\left[{w_1^\mathsf{T} L S_\theta^{-1} S_\theta S_\theta^{-1} L^\mathsf{T} w_1 }\right]\\
    &= \mathbb{E}_{\theta}{\mathsf{tr}({ L S_\theta^{-1} L^\mathsf{T} w_1 w_1^\mathsf{T} }})\\ 
	&= \mathsf{tr}({ L S_\theta^{-1} L^\mathsf{T} S_w }),
	\end{align*}
	where the second equality follows from Assumption~\ref{ass:linear:sys}\,(ii), which implies that~$x_0$ and~$w_1$ are independent, and from Assumption~\ref{ass:model:MDP:regularity}\,(ii), which implies that~$x_0$ is governed by the invariant state distribution~$\nu_\theta$ and thus has zero mean and covariance matrix~$S_\theta$. Substituting the resulting trace term into~\eqref{eq:rate:function:AR:MDP:proof} yields
	\begin{align*}
	\max_{L\in\mathbb{R}^{d\times d}} \left\{ \langle{\theta'- \theta},{L}\rangle - \frac{1}{2} \mathsf{tr}\left({ L S_\theta^{-1} L^\mathsf{T} S_w }\right)\right\} = \frac{1}{2} \mathsf{tr}\left({S_w^{-1}(\theta' - \theta) S_\theta (\theta' - \theta)^\mathsf{T}}\right)=I(\theta^\prime,  \theta),
	\end{align*}
	where the first equality follows from our analytical solution of problem~\eqref{LQG:proof:auxiliary:op} in the special case where $B_1=S_\theta^{-1}$, $B_2=S_2$ and $C=\theta'-\theta$. Thus, the rate function~\eqref{eq:rate:function:AR:MDP:proof} coincides indeed with the discrepancy function~$I(\theta^\prime,  \theta)$ defined in~\eqref{eq:rate:function:AR:MDP}.
	
	At last, we show that the moderate deviations principle established for $\|\theta\|_2<1$ remains valid for all asymptotically stable system matrices. To this end, fix any $\theta$ with $\rho(\theta)<1$. By standard Lyapunov stability theory, there exists $P\succ 0$ with $P-\theta^\mathsf{T} P\theta \succ 0$; see, {\em e.g.}, \autocite[Theorem~5.3.5]{ref:AREbook}. Using~$P$, we can apply the change of variables $\bar x_t=  P^{\frac{1}{2}}x_t$ and $\bar w_t=P^{\frac{1}{2}}w_t$ to obtain the auxiliary linear dynamical system
\[
    \bar x_{t+1} = \bar \theta\, \bar x_t + \bar w_t ,\quad \bar x_0\sim \bar \nu,
\]
with system matrix~$\bar \theta = P^{\frac{1}{2}}\theta P^{-\frac{1}{2}}$, where the noise $\bar w_t$ has mean zero and covariance matrix $S_{\bar w}=P^{\frac{1}{2}}S_w P^{\frac{1}{2}}$ for all $t\in\mathbb{N}$, and~$\bar \nu=\nu\circ P^{-\frac{1}{2}}$ is the pushforward distribution of~$\nu$ under the coordinate transformation~$P^{\frac{1}{2}}$. Note also that the invariante state covariance matrix is given by $S_{\bar \theta}= P^{\frac{1}{2}}S_\theta P^{\frac{1}{2}}$. By construction, the auxiliary linear system is equivalent to~\eqref{eq:LTI:system} and satisfies Assumptions~\ref{ass:linear:sys}\,(ii) and \ref{ass:model:MDP:regularity}. Moreover, multiplying $P-\theta^\mathsf{T} P\theta \succ 0$ from both sides with~$P^{-\frac{1}{2}}$ yields~$I_n - \bar \theta^\mathsf{T} \bar\theta\succ 0$, which means that the largest eigenvalue of $\bar \theta^\mathsf{T} \bar\theta$ is strictly smaller than~1 or, equivalently, that~$\|\bar \theta\|_2 < 1$. If we denote by $\widehat{\bar \theta}_T$ the least squares estimator for $\bar\theta$ based on~$T$ state observations of the auxiliary linear system, we may then conclude from the first part of the proof that the estimators $\{\sqrt{T / a_T}(\widehat{\bar \theta}_T - \theta) + \theta\}_{T\in\mathbb{N}}$ satisfy a moderate deviations principle with rate function
\[
    \bar I(\bar \theta',\bar \theta) = \mathsf{tr}\left(S_{\bar w}^{-1}(\bar \theta'-\bar \theta) S_{\bar \theta}(\bar \theta'-\bar \theta)^\mathsf{T} \right).
\] 
One also readily verifies from~\eqref{eq:LS:MDP} that the least squares estimators pertaining to the original and the auxiliary linear systems are related through the continuous transformation~$\widehat{\bar \theta}_T=P^{\frac{1}{2}}\widehat{\theta}_T P^{-\frac{1}{2}}$. The corresponding {\em transformed} estimators evidently obey the same relation. By the contraction principle~\autocite[Theorem~4.2.1]{dembo2009large}, the estimators $\{\sqrt{T / a_T}(\widehat {\theta}_T - \theta) + \theta\}_{T\in\mathbb{N}}$ thus satisfy a moderate deviations principle with rate function
\begin{align*}
    \bar I(P^{-\frac{1}{2}}\theta'P^{\frac{1}{2}},P^{-\frac{1}{2}} \theta P^{\frac{1}{2}}) = I(\theta',\theta).
\end{align*}
This observation completes the proof. 
\end{proof}

To prove that $I(\theta',\theta)$ is analytic, we recall that the stationary state covariance matrix $S_\theta$ is analytic in~$\theta\in\Theta$.

\begin{lemma}[Analyticity of $S_{\theta}$ {\autocite{polderman2}}, Lemma~3.2]
	\label{lem:dlyap_cont}
	If $S_w\succ 0$, then the solution $S_{\theta}$ to the Lyapunov equation \eqref{eq:Lyapunov} is analytic in $\theta\in\Theta$.
\end{lemma}

\begin{proof}[Proof of Proposition~\ref{prop:I_inf}]
    As for assertion~{\em (i)}, note that the quadratic function~$(\theta'-\theta)^\mathsf{T}(\theta'-\theta)S_w^{-1}$ is manifestly analytic in~$(\theta',\theta)$. Moreover, the stationary state covariance matrix $S_{\theta}$ is analytic in $\theta$ (and thus also in $(\theta',\theta)$) by virtue of Lemma~\ref{lem:dlyap_cont}. The rate function~$I(\theta',\theta)$ defined in~\eqref{eq:rate:function:AR:MDP} can therefore be viewed as an inner product of two matrix-valued analytic functions and is thus analytic thanks to~\autocite[Proposition~2.2.2]{ref:Real_Analytic_02}. 
    
    The proof of assertion~{\em (ii)} consists of two steps. We first prove that if a sequence $\{\theta_k\}_{k\in \mathbb{N}}$ in~$\Theta$ has an unstable limit~$\theta$ (\textit{i.e.}, $\rho(\theta)=1$), then there exists a subsequence $\{\theta_{k_l}\}_{l\in \mathbb{N}}$ with $\lim_{l\to \infty}I(\theta',\theta_{k_l})=\infty$ for all~$\theta'\in\Theta$ (Step~1). We then use this result to show that the set~$\{\theta\in\Theta:I(\theta', \theta)\leq r\}$ is compact for all~$r\ge 0$ (Step~2).

    {\em Step~1}: We first derive an easily computable lower bound on the rate function~$I(\theta',\theta)$ for any asymptotically stable matrices $\theta',\theta\in\Theta$. To this end, we denote by $\lambda\in\mathbb C$ an eigenvalue of~$\theta$ whose modulus~$|\lambda|$ matches the spectral radius~$\rho(\theta)<1$. We further denote by~$v\in\mathbb C^n$ a normalized eigenvector corresponding to the eigenvalue~$\lambda$, that is, $\|v\|=1$ and~$\theta v = \lambda v$. 
	We also use $\beta=\lambda_{\min}(S_w)/\lambda_{\max}(S_w)>0$ as a shorthand for the inverse condition number of the noise covariance matrix $S_w\succ 0$. Recalling that for any $A,B,C\in\mathcal{S}^n_{\succeq 0}$ the semidefinite inequality $A\succeq B$ implies~$\mathsf{tr}({AC})\geq \mathsf{tr}({BC})$, we find the following estimate.
	\begin{align*}
	2 I(\theta',\theta) &=  \mathsf{tr}\left({S_w^{-1}(\theta^\prime - \theta) S_\theta (\theta^\prime - \theta)^\mathsf{T}}\right) \\&\geq  \lambda^{-1}_{\max}(S_w)\,\mathsf{tr}\left({(\theta^\prime - \theta) S_\theta (\theta^\prime - \theta)^\mathsf{T}}\right) \\
	& \geq  \beta \sum_{k=0}^\infty \mathsf{tr}\left({(\theta^\prime - \theta) \theta^k (\theta^k)^\mathsf{T} (\theta^\prime - \theta)^\mathsf{T}}\right)\\ &  =  \beta \sum_{k=0}^\infty \mathsf{tr}\left({(\theta^k)^\mathsf{T} (\theta^\prime - \theta)^\mathsf{T}(\theta^\prime - \theta) \theta^k}\right)\\
	& \geq  \beta \sum_{k=0}^\infty v^{\mathsf{H}}(\theta^k)^\mathsf{T} (\theta^\prime - \theta)^\mathsf{T}(\theta^\prime - \theta) \theta^k v\\ &
	= \beta \sum_{k=0}^\infty |\lambda|^{2k}v^{\mathsf{H}} (\theta^\prime - \theta)^\mathsf{T}(\theta^\prime - \theta) v\\ 
	&=\beta \|  (\theta^\prime - \theta)v\|^2 \frac{1}{1-|\lambda|^2}
	\end{align*}
	Here, the first equality follows from the definition of the rate function in~\eqref{eq:rate:function:AR:MDP}, and the first inequality exploits the bound~$\lambda_{\max}(S_w)I_{n}\succeq S_w$. The second inequality holds due to the series representation $S_\theta = \sum_{t=0}^\infty \theta^t S_w (\theta^t)^\mathsf{T}$, the bound $S_w\succeq \lambda_{\min}(S_w)I_{n}$ and the definition of~$\beta$. The second equality exploits the cyclicity property of the trace, and the third inequality holds because any (real) matrix $C\in\mathcal{S}^n_{\succeq 0}$ satisfies
	\begin{equation*}
	\mathsf{tr}({C}) \ge  w^{\mathsf{H}} C w \quad \forall w\in\mathbb C^n:\;\|w\|=1.
	\end{equation*}
	The third equality then uses the eigenvalue equation~$\theta v = \lambda v$, and the last equality holds because~$|\lambda|=\rho(\theta)<1$. We thus conclude that the rate function admits the lower bound
	\begin{equation} \label{eq:proof:coercivity:step1}
	I(\theta',\theta)
	\geq \frac{\beta}{2}\|  (\theta^\prime - \theta)v\|^2 \frac{1}{1-|\lambda|^2}.
	\end{equation}
	Consider now a converging sequence~$\{\theta_k\}_{k\in \mathbb{N}}$ in $\Theta$ whose limit $\theta$ satisfies $\rho(\theta)=1$. Define $\lambda_k\in\mathbb C$ as an eigenvalue of $\theta_k$ with $|\lambda_k|=\rho(\theta_k)<1$ and let $v_k\in\mathbb C^n$ be a normalized eigenvector corresponding to $\lambda_k$, that is, $\|v_k\|=1$ and~$\theta_k v_k = \lambda_k v_k$. As the spectral radius is a continuous function, we then have
	\[
	\lim_{k\to \infty}|\lambda_k| = \lim_{k\to \infty} \rho(\theta_k)=\rho(\lim_{k\to \infty}\theta_k)=\rho(\theta)=1.
	\]
	In addition, as the unit spheres in~$\mathbb C$ and in~$\mathbb C^n$ are both compact, there exists a subsequence $\{(\lambda_{k_l},v_{k_l})\}_{l\in\mathbb{N}}$ converging to a point~$(\lambda,v)\in \mathbb C\times \mathbb C^n$ with $|\lambda|=1$ and $\|v\|=1$. This limit satisfies the eigenvalue equation 
	\begin{equation} \label{eq:helping:coercivity:proof}
	\theta v = \lim_{l\to\infty} \theta_{k_l}v_{k_l}=\lim_{l\to\infty} \lambda_{k_l}v_{k_l} = \lambda v,
	\end{equation}
	which implies that $v$ is an eigenvector of $\theta$ corresponding to the eigenvalue $\lambda$ with $|\lambda|=1=\rho(\theta)$.
	
	The above reasoning allows us to conclude that 
	\begin{align*}
        \lim_{l\to\infty}I(\theta',\theta_{k_l}) &\geq \lim_{l\to\infty} \frac{\beta}{2}\|  (\theta_{k_l}-\theta')v_{k_l}\|^2 \frac{1}{1-|\lambda_{k_l}|^2}\\
        & =\lim_{l\to\infty}\frac{\beta}{2}\|\lambda_{k_l}v_{k_l} - \theta' v_{k_l}\|^2 \frac{1}{1-|\lambda_{k_l}|^2} \\ &=\frac{\beta}{2}\|\lambda v - \theta' v\|^2 \lim_{l\to\infty}\frac{1}{1-|\lambda_{k_l}|^2}=\infty,
	\end{align*}
	where the inequality follows from \eqref{eq:proof:coercivity:step1}, the first equality holds because $\theta_k v_k = \lambda_k v_k$, and the second equality exploits~\eqref{eq:helping:coercivity:proof}. Finally, the last equality holds because $\lim_{k\to\infty}|\lambda_k|=1$ and because the term $\frac{\beta}{2}\|\lambda v - \theta' v\|^2$ is strictly positive. Indeed, this non-negative term can only vanish  if~$\theta'v=\lambda v$, which would imply that $\theta'$ is unstable (as $|\lambda|=1)$ and thus contradict the assumption that~$\theta'\in\Theta$. This observation completes Step~1.
	
	{\em Step~2}: Select now any~$\theta'\in\Theta$	and~$r\ge 0$, and define $\mathcal{A}=\{\theta\in\Theta : I(\theta',\theta)\leq r\}$. In order to prove that $\mathcal{A}$ is compact, we need to show that it is bounded and closed. This is potentially difficult because~$\Theta$ itself is unbounded and open. In order to prove boundedness of~$\mathcal{A}$, note that every $\theta\in\mathcal{A}$ satisfies
	\begin{align*}
	r\geq I(\theta',\theta) &= \frac{1}{2} \mathsf{tr}\left({S_w^{-1}(\theta^\prime - \theta) S_\theta (\theta^\prime - \theta)^\mathsf{T}}\right)\\
	&\geq \frac{1}{2} \mathsf{tr}\left({S_w^{-1}(\theta^\prime - \theta) S_w (\theta^\prime - \theta)^\mathsf{T}}\right),
	\end{align*}
	where the second inequality follows from the trivial bound~$S_{\theta} \succeq S_w$, which is implied by the Lyapunov equation \eqref{eq:Lyapunov}. 
	Thus, the sublevel set~$\mathcal{A}$ is contained in a bounded ellipsoid,
	\begin{equation*}
	\mathcal{A}\subset \left\{\theta\in\mathbb{R}^{n\times n}: \frac{1}{2} \mathsf{tr}\left({S_w^{-1}(\theta^\prime - \theta) S_w (\theta^\prime - \theta)^\mathsf{T}}\right)\leq r\right\},
	\end{equation*}
	and thus~$\mathcal{A}$ is bounded. To show that~$\mathcal{A}$ is closed, consider a converging sequence $\{\theta_k\}_{k\in\mathbb{N}}$ in~$\mathcal{A}$ with limit~$\theta$. We first prove that~$\theta\in\Theta$. Suppose for the sake of argument that $\theta\notin\Theta$. As $\theta$ is the limit of a sequence in~$\mathcal{A}\subset\Theta$, this implies that~$\theta$ must reside on the boundary of~$\Theta$ (\textit{ i.e.}, $\rho(\theta)=1$). By the results of Step~1, we may thus conclude that there exists a subsequence~$\{\theta_{k_l}\}_{l\in\mathbb{N}}$ with~$\lim_{l\to\infty}I(\theta',\theta_{k_l}) = \infty$. Clearly, we then have~$I(\theta',\theta_{k_l})>r$ for all sufficiently large $l$, which contradicts the assumption that~$\theta_{k_l}\in \mathcal{A}$ for all~$l\in\mathbb{N}$. Thus, our initial hypothesis was wrong, and we may conclude that~$\theta\in\Theta$. In addition, we have
	\begin{align*}
	r\geq \lim_{k\to\infty}I(\theta',\theta_k) =  I(\theta',\lim_{k\to\infty}\theta_k) = I(\theta',\theta),
	\end{align*}
	where the inequality holds because~$\theta_k\in\mathcal{A}$ for all~$k\in\mathbb{N}$. Here, the first equality follows from assertion~{\em (i)}, which ensures that the rate function is analytic and thus continuous. Hence, we find that~$\theta\in\mathcal{A}$. As the sequence~$\{\theta_k\}_{k\in\mathbb{N}}$ was chosen arbitrarily, we conclude that~$\mathcal{A}$ is closed. In summary, we have shown that~$\mathcal{A}$ is bounded and closed and thus compact. This observation completes Step~2. Hence, assertion~{\em (ii)} follows.
	
	As for assertion~{\em (iii)}, fix~$\theta'\in \Theta$ and consider a sequence~$\{\theta_k\}_{k\in\mathbb{N}}$ in~$\Theta$ whose limit~$\theta$ resides on the boundary of the open set~$\Theta$. This implies that~$\theta\notin\Theta$. Next, choose any~$r\ge 0$. We know from assertion~{\em (ii)} that~$\mathcal{A}=\{\theta\in\Theta : I(\theta',\theta)\leq r\}$ is a compact subset of~$\Theta$, and thus~$\theta\notin\mathcal{A}$. Hence, the complement of~$\mathcal{A}$ represents an open neighborhood of~$\theta$, and thus there exists~$k(r)\in\mathbb{N}$ such that~$\theta_k\notin\mathcal{A}$ and~$I(\theta', \theta_k)\ge r$ for all~$k\ge k(r)$. As~$r$ was chosen freely, this means that~$\lim_{k\rightarrow \infty} I(\theta', \theta_k)=\infty$.
\end{proof}

\begin{proof}[Proof of Lemma~\ref{lem:rate:to:norm}]
By the definition of the rate function we have
\begin{align*}
    & 2I(\theta',\theta) = \mathsf{tr}\left(S_w^{-1}(\theta'-\theta)S_{\theta}(\theta'-\theta)^\mathsf{T} \right)\geq  \sigma_{\mathrm{min}}(S_w^{-1})\sigma_{\mathrm{min}}(S_{\theta})\|\theta'-\theta\|_F^2 \geq  \frac{1}{\kappa(S_w)} \|\theta'-\theta\|_2^2,
\end{align*}
where the third inequality holds because~$S_{\theta}\succeq S_w$ and $\sigma_{\mathrm{min}}(S_w^{-1})=1/\sigma_{\mathrm{max}}(S_w)$. The claim then follows by multiplying the above inequality with~$2\kappa(S_w)$ and taking square roots on both sides.
\end{proof}

\subsection{Proofs of Section~\ref{ssec:stability:guarantees}}
\begin{proof}[Proof of Proposition~\ref{prop:consistency}]
Recall that $\lim_{T\to\infty}\widehat\theta_T=\theta$ $\mathbb{P}_\theta$-almost surely \autocite{ref:Campi_98}. Therefore, we have $\mathbb{P}_\theta$-almost surely that
\begin{align*}
    \lim_{T\to\infty}\mathcal{P}(\widehat{\theta}_T)
    &=\lim_{T\to\infty}\arg\min_{\bar\theta\in\Theta}I(\widehat\theta_T,\bar\theta) \\
    &=\arg\min_{\bar\theta\in\Theta}\lim_{T\to\infty}I(\widehat\theta_T,\bar\theta)\\
    &=\arg\min_{\bar\theta\in\Theta}I\left(\lim_{T\to\infty}\widehat\theta_T,\bar\theta\right)\\
    &=\arg\min_{\bar\theta\in\Theta}I(\theta,\bar\theta) =\theta,
\end{align*}
where the first equality exploits the definition of $\mathcal{P}(\widehat{\theta}_T)$ in~\eqref{equ:rev:I:proj}. The second equality follows from the strict convexity of the rate function in its first argument and \autocite[Theorem~9.17]{sundaram_1996}, which imply that the reverse $I$-projection is continuous. The third equality follows from the continuity of the rate function established in Proposition~\ref{prop:I_inf}~\eqref{prop:I_inf:cont}, and the last equality holds because the rate function vanishes if and only if its arguments coincide. This proves the proposition. 
\end{proof}

{
\begin{proof}[Proof of Lemma~\ref{lem:learn:stab}]
Lemma~\ref{lem:rate:to:norm} and the monotonicity of the square root function imply that $\|\widehat{\theta}_T-\mathcal{P}(\widehat{\theta}_T)\|_2 \leq \widehat{\varepsilon}_T$ is a $\mathbb P_\theta$-almost sure event; see also the discussion following Lemma~\ref{lem:rate:to:norm}. As for assertion~{\em (i)},
we thus have
\begin{align*}
    \mathbb{P}_{\theta}&\left(\|\theta-\mathcal{P}(\widehat{\theta}_T)\|_2\leq \varepsilon+ \widehat\varepsilon_T \right)\\
    &\geq \mathbb{P}_{\theta}\left(\|\theta-\widehat{\theta}_T\|_2 + \|\widehat{\theta}_T-\mathcal{P}(\widehat{\theta}_T)\|_2\leq \varepsilon+ \widehat\varepsilon_T \right)\\
    &\geq \mathbb{P}_{\theta}\left(\|\theta-\widehat{\theta}_T\|_2\leq \varepsilon ,~ \|\widehat{\theta}_T-\mathcal{P}(\widehat{\theta}_T)\|_2 \leq \widehat \varepsilon_T\right)\\
    &= \mathbb{P}_{\theta}\left( \|\theta-\widehat{\theta}_T\|_2\leq \varepsilon \right).
\end{align*}
To estimate the probability that the least squares estimator~$\widehat \theta_T$ differs from~$\theta$ at most by~$\varepsilon$ in the operator norm, we may leverage tools developed in~\autocite[\S~6]{ref:sarkar19a}. To this end, assume first that the noise is isotropic, \textit{i.e.}, assume that $S_w=\alpha I_n$ for some $\alpha>0$. In this case, \autocite[Theorem~1]{ref:sarkar19a} implies that $\mathbb{P}_{\theta}( \|\theta-\widehat{\theta}_T\|_2\leq \varepsilon )\geq 1-\beta$ for all~$\beta,\varepsilon\in (0,1)$ and sample sizes $T\geq \widetilde{O}(n) \log(1/\beta)/\varepsilon^2$. As~$\kappa(S_w)=\kappa(\alpha I_n)=1$, this settles assertion~{\em (i)} when the noise is isotropic.

Assume now that the noise is anisotropic with an arbitrary convariance matrix $S_w\succ 0$. The change of coordinates $\bar x_t=  S_w^{-\frac{1}{2}}x_t$ and $\bar w_t=S_w^{-\frac{1}{2}}w_t$ then yields the auxiliary system
\[
    \bar x_{t+1} = \bar \theta\, \bar x_t + \bar w_t ,\quad \bar x_0\sim \nu\circ S_w^{\frac{1}{2}},
\]
with~$\bar \theta = S_w^{-\frac{1}{2}}\theta S_w^{\frac{1}{2}}$ and isotropic noise $\bar w_t$ having zero mean and unit covariance matrix for all $t\in\mathbb{N}$. Denoting by $\widehat{\bar\theta}_T$ the least squares estimator for the auxiliary system, we find
\begin{align*}
    \mathbb{P}_{\theta}\left( \|\theta-\widehat{\theta}_T\|_2\leq \varepsilon \right)& = \mathbb{P}_{\theta}\left( \|S_w^{\frac{1}{2}}(\bar\theta-\widehat{\bar\theta}_T) S_w^{-\frac{1}{2}} \|_2\leq \varepsilon \right)\\
    & \geq \mathbb{P}_{\theta}\left( \|S_w^{\frac{1}{2}}\|_2\|\bar\theta-\widehat{\bar\theta}_T\|_2\| S_w^{-\frac{1}{2}} \|_2\leq \varepsilon \right)\\
    & = \mathbb{P}_{\theta}\left( \|\bar\theta-\widehat{\bar\theta}_T\|_2\leq \varepsilon \kappa(S_w)^{-\frac{1}{2}}\right),
\end{align*}
where the last equality holds because 
\[
    \|S_w^{{1}/{2}}\|_2\|S_w^{-{1}/{2}}\|_2=\kappa(S_w^{{1}/{2}})=\kappa(S_w)^{{1}/{2}}.
\]
From the first part of the proof for linear systems driven by isotropic noise we know that the resulting probability is no less than~$1-\beta$ whenever $T\geq \kappa(S_w) \widetilde{O}(n) \log(1/\beta)/\varepsilon^2$. This observation completes the proof of assertion~{\em (i)}.

The proof of assertion~{\em (ii)} first parallels that of assertion~{\em (i)}. In particular, multiplying $\varepsilon$ with $\sqrt{a_T/T}$ yields
\begin{align*}
    &\mathbb{P}_{\theta}\left(\|\theta-\mathcal{P}(\widehat{\theta}_T)\|_2\leq \varepsilon\sqrt{a_T/T}+ \widehat\varepsilon_T \right) = \mathbb{P}_{\theta}\left( \|\theta-\widehat{\theta}_T\|_2\leq \varepsilon \sqrt{a_T/T} \right).
\end{align*}
However, now we use the moderate deviations principle from Section~\ref{sec:moderate:deviations:theory} to bound the resulting probability. To this end, define $\mathcal{D}=\{\theta'\in \mathbb{R}^{n\times n}:\|\theta'-\theta\|_2> \varepsilon\}$. By Lemma~\ref{lem:rate:to:norm}, we have $I(\theta',\theta)> \varepsilon^2/(2\kappa(S_w))$ for any estimator realization $\theta'\in\mathcal D$, and thus $\inf_{\theta'\in \mathsf{cl} \, \mathcal{D}}I(\theta',\theta)\geq \varepsilon^2/(2\kappa(S_w))$. Recall now from Proposition~\ref{prop:MDP:closed:loop} that the transformed least squares estimators $\widehat{\vartheta}_{T}=\sqrt{T/a_T}(\widehat{\theta}_{T}-\theta)+\theta$ obey a moderate deviations principle with rate function~$I$.  
Hence, we have
\begin{align*}
     \limsup\limits_{T\to \infty}\frac{1}{a_T}\log \mathbb{P}_{\theta}\left(\|\widehat{\theta}_{T}-\theta\|_2 >  
    \varepsilon\sqrt{a_T/T}\right)=& \limsup\limits_{T\to \infty}\frac{1}{a_T}\log \mathbb{P}_{\theta}(\widehat{\vartheta}_{T}\in \mathcal{D})\\
    \leq& -\inf_{\theta'\in \mathsf{cl} \, \mathcal{D}}I(\theta',\theta) \leq- \varepsilon^2/(2\kappa(S_w)),
\end{align*}
where the equality exploits the definitions of~$\mathcal D$ and~$\widehat\vartheta_T$, and the first inequality follows from Proposition~\ref{prop:MDP:closed:loop}. By passing over to complementary events, we therefore obtain
\begin{equation*}
     \mathbb{P}_{\theta}\left(\|\widehat{\theta}_{T}-\theta\|_2 \leq 
     \varepsilon\sqrt{a_T/T}\right) \geq
     1 - e^{-\varepsilon^2 a_T/(2\kappa(S_w))+o(a_T)}.
\end{equation*}
For all sufficiently large sample sizes~$T$ satisfying the inequality $a_T \geq 2\kappa(S_w)(\log (1/\beta)+o(a_T))/\varepsilon^2$ this implies that
\begin{equation*}
     \mathbb{P}_{\theta}\left(\|\widehat{\theta}_{T}-\theta\|_2 \leq \varepsilon \sqrt{a_T/T}\right) \geq 1 - \beta.  
\end{equation*}
This observation completes the proof of assertion~{\em (ii)}.
\end{proof}
The proof of Lemma~\ref{lem:learn:stab} reveals that the requirement $\varepsilon\in (0,1)$ could be relaxed to $\varepsilon>0$ in assertion~{\em (ii)}. 
}

{
\begin{proof}[Proof of Proposition~\eqref{prop:finite:sample}]

As $\theta$ is $(\tau,\rho)$-stable, the defining properties of the reverse $I$-projection imply that
\begin{align*}
    I(\widehat{\theta}_T,\mathcal{P}(\widehat{\theta}_T))\leq & I(\widehat{\theta}_T,\theta)\\
    = & \frac{1}{2}\mathsf{tr}\left(S_w^{-1}(\widehat{\theta}_T-\theta)S_{\theta}(\widehat{\theta}_T-\theta)^{\mathsf{T}}\right)\\
    \leq & \frac{1}{2}\mathsf{tr}(S_w^{-1})\|\widehat{\theta}_T-\theta\|_2^2 \|S_{\theta}\|_2\\
    \leq & \frac{1}{2}n\kappa(S_w) \|\widehat{\theta}_T-\theta\|_2^2 \frac{\tau^2}{1-\rho^2},
\end{align*}
where the second inequality holds because $\mathsf{tr}(AB)\leq \mathsf{tr}(A)\|B\|_2$ for any two symmetric matrices $A,B \in \mathbb{R}^{n\times n}$, while the third inequality follows from~\cite[Proposition~E.5]{ref:Krauth-19}. 
Hence, up to problem-dependent constants, $I(\widehat{\theta}_T,\mathcal{P}(\widehat{\theta}_T))$ decays as least as fast as $\|\widehat{\theta}_T-\theta\|_2^2$. 
Combining the above estimate with Lemma~\ref{lem:rate:to:norm} and taking square roots then yields
\begin{align}
    \label{eq:rho-tau-estimate}
   \|\widehat{\theta}_T-\mathcal{P}(\widehat{\theta}_T)\|_2\leq  \|\widehat{\theta}_T-\theta\|_2 \kappa(S_w) \frac{n^\frac{1}{2}\tau}{\sqrt{1-\rho^2}}. 
\end{align}
Setting $\eta=\kappa(S_w) n^\frac{1}{2}\tau/\sqrt{1-\rho^2}\geq 1$, we may use a similar reasoning as in the proof Lemma~\ref{lem:learn:stab} to obtain
\begin{align*}
    \mathbb{P}_{\theta}&\left(\|\theta-\mathcal{P}(\widehat{\theta}_T)\|_2\leq 2\eta\varepsilon \right)\\
    &\geq \mathbb{P}_{\theta}\left(\|\theta-\widehat{\theta}_T\|_2\leq \eta \varepsilon ,~ \|\widehat{\theta}_T-\mathcal{P}(\widehat{\theta}_T)\|_2 \leq \eta \varepsilon\right)\\
    &\geq \mathbb{P}_{\theta}\left(\|\theta-\widehat{\theta}_T\|_2\leq \varepsilon ,~ \|\widehat{\theta}_T-\mathcal{P}(\widehat{\theta}_T)\|_2 \leq \eta \|\theta-\widehat{\theta}_T\|_2\right)\\
    &= \mathbb{P}_{\theta}\left( \|\theta-\widehat{\theta}_T\|_2\leq \varepsilon \right),
\end{align*}
where the second inequality holds because~$\eta\geq 1$, the equality follows from~\eqref{eq:rho-tau-estimate}, which holds with certainty. However, from the proof of Lemma~\eqref{lem:learn:stab}\,{\em{i}} we already know that $\mathbb{P}_{\theta}( \|\theta-\widehat{\theta}_T\|_2\leq \varepsilon)\geq 1-\beta$ whenever $T\geq \kappa(S_w) \widetilde{O}(n) \log(1/\beta)/\varepsilon^2$. This observation completes the proof of assertion~{\em (i)}. 

The proof of assertion~{\em (ii)} widely parallels that of assertion~{\em (i)} and is thus omitted for brevity. 
\end{proof}
}

\subsection{Proofs of Section~\ref{sec:num:I:proj}}

\begin{proof}[Proof of Proposition~\ref{prop:proxy}]
The claim follows immediately from the discussion leading to Proposition~\ref{prop:proxy}.
\end{proof}

The approximate computation of the reverse $I$-projection exploits standard results on infinite-horizon dynamic programming (see, \textit{e.g.}, \autocite[Chapter 3]{ref:Basar_95} or~\autocite{ref:BertVol2_07}) as well as the following exact constraint relaxation result borrowed from \autocite[Lemma A-0.1]{ref:mscThesis_wjongeneel_2019}; see also \autocite{ref:RLQR_CDC_19}. We repeat this result here to keep the paper self-contained.

\begin{lemma}[Exact constraint relaxation]
	\label{lem:relax_max}
	Let $f$ and $g$ be two arbitrary functions from $\Theta$ to $(-\infty,\infty]$, and consider the two closely related {minimization} problems
	\begin{align*}
	\mathcal{P}_1(r)~&:~ \inf_{\theta \in \Theta} \left\{ f(\theta) : g(\theta) \le r\right\}\\
	\mathcal{P}_2(\delta)~&:~ \inf_{\theta \in \Theta} f(\theta) + \frac{g(\theta)}{\delta}
	\end{align*}
	parametrized by $r \in\mathbb{R}$ and $\delta\in(0,\infty)$, respectively. If the penalty-based minimization problem $\mathcal{P}_2(\delta)$ admits an optimal solution $\theta^\star_2(\delta)$ for the parameter values $\delta$ within some set $\Delta\subset (0,\infty)$, then the following hold.
	\begin{enumerate}[(i)]
		\item \label{lem:relax:monotone}
		The function $ h(\delta) = g(\theta^\star_2(\delta))$ is non-decreasing in the parameter~$\delta \in\Delta$.
		
		\item \label{lem:relax:opt:1}
		If there exits $\delta\in\Delta$ with $h(\delta) = r$, then the constrained minimization problem~$\mathcal{P}_1(r)$ is solved by $\theta^\star_2(\delta)$.
	\end{enumerate}
\end{lemma}

\begin{proof}
	In order to prove assertion~{\em (i)}, choose any parameters $\delta_1,\delta_2 \in\Delta$ with $\delta_1 > \delta_2$. As $\theta^\star_2(\delta_1)$ is optimal in $\mathcal{P}_2(\delta_1)$ and $\theta^\star_2(\delta_2)$ is optimal in $\mathcal{P}_2(\delta_2)$, one can readily verify that
	\begin{align*}
	f\big(\theta^\star_2(\delta_1)\big) + \frac{g(\theta^\star_2(\delta_1))}{\delta_1}\leq f(\theta^\star_2(\delta_2)) + \frac{g(\theta^\star_2(\delta_2))}{\delta_1}
	\end{align*}
	and
	\begin{align*}
	f(\theta^\star_2(\delta_2)) + \frac{g(\theta^\star_2(\delta_2))}{\delta_2}
	\leq f\big(\theta^\star_2(\delta_1)\big) + \frac{g(\theta^\star_2(\delta_1))}{\delta_2}.
	\end{align*}
	Summing up these two inequalities yields
	\begin{align*}
	 \left(\frac{1}{\delta_2}-\frac{1}{\delta_1}\right)g\big(\theta^\star_2(\delta_2)\big) \le \left(\frac{1}{\delta_2}-\frac{1}{\delta_1}\right)g\big(\theta^\star_2(\delta_1)\big)  
	 \iff h(\delta_2)= g\big(\theta^\star_2(\delta_2)\big) \le g\big(\theta^\star_2(\delta_1)\big) =h(\delta_1),
	\end{align*}
	where the equivalence holds because $\delta_1>\delta_2$. This completes the proof of assertion~{\em (i)}. 
	
	As for assertion~{\em (ii)}, fix any $r\in\mathbb R$ and assume that there exists $\delta\in\Delta$ with $r=h(\delta)$. We need to show that the optimizer~$\theta^\star_2(\delta)$ of~$\mathcal{P}_2(\delta)$ is also optimal in $\mathcal{P}_1(r)$.
	To this end, observe that~$\theta^\star_2(\delta)$ is feasible in~$\mathcal{P}_1(r)$ because $r=h(\delta)=g(\theta_2^\star(\delta))$. It then suffices to prove optimality. Assume for the sake of contradiction that there exists $\theta'_1\in\Theta$ with $f(\theta'_1) < f(\theta^\star_2(\delta))$ and $g(\theta'_1) \le g(\theta^\star_2(\delta))=r$. In this case, we have
	\begin{align*}
	f(\theta'_1) + \frac{g(\theta'_1)}{\delta} <  f\big(\theta^\star_2(\delta)\big) + \frac{g(\theta^\star_2(\delta))}{\delta} \,,
	\end{align*}
	which contradicts the optimality of $\theta^\star_2(\delta)$ in~$\mathcal{P}_2(\delta)$. We thus conclude that $\theta^\star_2(\delta)$ must indeed solve $\mathcal{P}_1(r)$. 
\end{proof}

We also recall the following matrix inversion lemma.

\begin{lemma}[Matrix inversion~{\autocite[p.~19]{Verhaegen}}]
\label{lem:M:inv}
If $A$ and $C$ are invertible matrices, then we have
\begin{align*}
    &(A+BCD)^{-1} =A^{-1}-A^{-1}B(C^{-1}+DA^{-1}B)^{-1}DA^{-1}.
\end{align*}
\end{lemma}

\begin{proof}[Proof of Proposition~\ref{prop:opt:sol}]
Fix any $\theta'\notin \Theta$,
{ and identify the reverse $I$-projection problem~\eqref{equ:jr} with problem~$\mathcal{P}_1(r)$ from Lemma~\ref{lem:relax_max}, that is, set $f(\theta)=\mathsf{tr}(Q S_{\theta})$ and $g(\theta)=I(\theta',\theta)$. By the definition of the rate function~$I$ in~\eqref{eq:rate:function:AR:MDP}, the corresponding unconstrained problem~$\mathcal{P}_2(\delta)$ is equivalent to
\begin{align*}
    \min_{\theta\in \Theta} \mathsf{tr}(Q S_{\theta})+\frac{1}{\delta} I(\theta',\theta)
    &= \min_{\theta\in \Theta} \lim_{T\to \infty}\frac{1}{T} \textstyle \mathbb E_\theta\Big[\sum_{k=0}^{T-1} x_k^{\mathsf{T}}Qx_k  + \frac{1}{2\delta}x_k^{\mathsf{T}}(\theta'-\theta)^{\mathsf{T}} S_w^{-1}(\theta'-\theta) x_k \Big]\\
    &= \min_{L\in \mathbb R^{n\times n}} \lim_{T\to \infty}\frac{1}{T} \textstyle \mathbb E_{\theta'+L}\Big[\sum_{k=0}^{T-1} x_k^{\mathsf{T}}\left(Q + \frac{1}{2\delta}L^{\mathsf{T}} S_w^{-1}L\right) x_k \Big],
\end{align*}
where the first equality exploits the Markov law of large numbers. The second equality follows from the variable substitution $L\leftarrow \theta'-\theta$. Note that the constraint $\theta'+L\in\Theta$ can be relaxed because $\mathbb E_{\theta'+L}[x_k^{\mathsf{T}}Qx_k]$ diverges with $k$ whenever $\theta'+L$ is unstable. Indeed, in this case the trace of the covariance matrix of~$x_k$ explodes. Also, we remark again that $\mathbb{E}_{\theta'+L}[\cdot]$ merely indicates that the distribution is parametric in $\theta'+L$, the variables $\theta'$ and $L$ are not random in the above. 

As any infinite horizon time-homogeneous LQR problem with average cost criterion is solved by a linear control policy of the form $u_k=Lx_k$ for some~$L\in\mathbb R^{n\times n}$, problem~$\mathcal{P}_2(\delta)$ is equivalent to
\begin{equation}
\label{equ:best:DP}
    \begin{array}{cl}
        \min\limits_{\varphi_k(\cdot)} & \displaystyle\lim_{T\to \infty}\frac{1}{T}
        \mathbb{E} \left[\textstyle\sum^{T-1}_{k=0}  x_k^{\mathsf{T}} Q x_k + u_k^{\mathsf{T}} R u_k \right]  \\
         \mathrm{s.t.}&   x_{k+1}=\theta'x_k +u_k +w_k,\quad x_0\sim \nu ,\\ & u_k=\varphi_k(x_k),
    \end{array}
\end{equation}
where $R=(2 \delta S_w)^{-1}$, and the expectation is evaluated with respect to the canonical probability measure induced by the initial state distribution~$\nu$, the control policy $\{\varphi_k\}_{k=0}^\infty$ and the corresponding system dynamics. As standard stabilizability and detectability assumptions are trivially satisfied~\autocite[Chapter 4]{ref:BertVolI_05}, the LQR problem~\eqref{equ:best:DP} is solvable for every~$\delta>0$. Its optimal solution is a stationary linear control policy with state feedback gain $L_{\delta}= -({P}_{\delta}+R)^{-1}{P}_{\delta }\theta'$, where ${P}_{\delta }$ is the unique positive definite solution of the Riccati equation
\begin{equation*}
    P_{\delta} = Q + \theta'^{\mathsf{T}} P_{\delta} \theta' - \theta'^{\mathsf{T}}P_{\delta}( P_{\delta} + R)^{-1}P_{\delta}\theta'.
\end{equation*}
Note that this equation is equivalent to~\eqref{equ:P_ARE:best} by Lemma~\ref{lem:M:inv} and the definition of~$R$. Hence, problem $\mathcal P_2(\delta)$ is solved by
\begin{align*}
    \theta^\star_2(\delta)&=\theta'+L_\delta =\theta'-({P}_{\delta}+R)^{-1}{P}_{\delta }\theta'\\ &  = (I_n+R^{-1} P_{\delta})^{-1}\theta'
\end{align*}
for any $\delta>0$, where the third equality follows again from Lemma~\ref{lem:M:inv}. Note that the last expression is equivalent to~$\theta^\star_\delta$ from the proposition statement. By using \autocite[Lemma 3.2]{polderman2}, one can show that $P_{\delta}$ and consequently also~$\theta^\star_\delta$ are real-analytic in $\delta>0$. As the rate function~$I$ is analytic thanks to Proposition~\ref{prop:I_inf}\,{\em (i)}, the function~$\varphi^{-1}(\delta)=I(\theta',\theta^\star_\delta)$ is thus analytic as a composition of two analytic functions. In addition, $\varphi^{-1}(\delta)$ is non-decreasing thanks to Lemma~\ref{lem:relax_max}\,{\em (i)}. As any non-decreasing analytic function that is not constant must be strictly monotonically increasing, we may conclude that~$\varphi^{-1}:(0,\infty)\to(\underline r, \overline r)$ is bijective, where 
\[
    \underline{r}=\lim_{\delta\downarrow 0}\varphi^{-1}(\delta)\quad \text{and} \quad \overline{r}=\lim_{\delta \uparrow  \infty}\varphi^{-1}(\delta). 
\]
Note that if $\delta$ tends to~$0$, then problem~$\mathcal P_2(\delta)$ just minimizes $g(\theta)=I(\theta',\theta)$ over~$\Theta$, in which case the reverse $I$-projection $\mathcal P(\theta')$ is optimal. Recall that $\mathcal P(\theta')$ is also optimal in problem $\mathcal P_1(r)$ for $r=I(\theta',\mathcal P(\theta'))=\underline r$. Note also that if $\delta$ tends to $\infty$, then problem~$\mathcal P_2(\delta)$ just minimizes $f(\theta)=\mathsf{tr}(Q S_{\theta})$ over~$\Theta$, in which case the trivial solution~$\theta=0$ is optimal. Clearly, $0$ is also optimal in problem $\mathcal P_1(r)$ for $r=I(\theta',0)=\overline r$. Hence, we may define~$\varphi:(\underline r, \overline r)\to (0,\infty)$ as the inverse of~$\varphi^{-1}$. By construction, $\varphi$ is analytic, strictly increasing and bijective.

In summary, Lemma~\ref{lem:relax_max} implies that for each $r\in (\underline{r},\overline{r})$ we may set $\delta=\varphi(r)$  such that $\theta^\star_\delta$ is the unique optimal solution of problem~$\mathcal P_1(r)$, and this solution satisfies $I(\theta',\theta^\star_\delta)=r$.}
\end{proof}

\begin{proof}[Proof of Proposition~\ref{prop:numerical:comp}]
Fix any $\theta'\notin \Theta$. Proposition~\ref{prop:jr:stable} implies that $\lim_{\delta\downarrow 0}\theta^{\star}_{\delta}=\mathcal{P}(\theta')$. 
However, one cannot evaluate~$\theta^{\star}_{\delta}$ at $\delta= 0$. Indeed, the Riccati equation~\eqref{equ:P_ARE:best} fails to have a positive definite solution for $\delta=0$ because~$\theta'$ is unstable. Nevertheless, the error bound~\eqref{equ:poly:error} follows directly from the Pinsker-type inequality established in Lemma~\ref{lem:rate:to:norm} and from the analyticity of $I(\theta',\theta^{\star}_{\delta})$ in $\delta\in(0,\infty)$. For $\delta>0$, the proof of Proposition~\ref{prop:opt:sol} reveals that $\theta^{\star}_{\delta}$ can be computed by solving problem~\eqref{equ:best:DP}, which can be addressed with standard LQR routines. Hence, the computational bottleneck is the solution of the Riccati equation~\eqref{equ:P_ARE:best}. The state-of-the-art methods to solve~\eqref{equ:P_ARE:best} utilize a QZ algorithm that has time and memory complexity of the order $O(n^3)$ and ${O}(n^2)$, respectively; see, {\em e.g.}, \autocite{ref:Pappas} and \autocite[Algorithm 7.7.3]{ref:GvL}. However, large problem instances should be addressed with alternative schemes such as the ones proposed in~\autocite{ref:Laub:parallel,ref:Fassbender}.
\end{proof}

\begin{proof}[Proof of Corollary~\ref{cor:struc}]
Fix any $\theta'\in \Theta'$. In view of Proposition~\ref{prop:opt:sol}, it follows directly from~\autocite{ref:Topo2020} that for any $\delta\in (0,\infty)$ there is a $\Lambda_{\delta}$ with $\mathrm{det}(\Lambda_{\delta})>0$ such that $\theta^{\star}_{\delta}=\Lambda_{\delta}^{-1}\theta'$. Next, as $\Lambda_{\delta}=(I_n+2\delta S_w P_{\delta})$, $S_w\succ 0$, $P_{\delta}\succeq Q\succ 0$, $P_{\delta}$ is real-analytic in $\delta$ over $(0,\infty)$ and $\lambda_i(S_w P_{\delta})>0$ for $i\in [n]$ we have that $\lim_{\delta \downarrow 0}\mathrm{det}(\Lambda_{\delta})>0$, which concludes the proof. 
\end{proof}

\subsection{Additional numerical results to Section~\ref{sec:num}}
Figure~\ref{fig:learn:stab:random:app} provides additional numerical results for the examples in Section~\ref{sec:num}. {We point out that the gap between the spectral radii of~$\widehat \theta_T$ and~$\mathcal{P}(\widehat{\theta}_T)$ is more pronounced in Figure~\ref{fig:m64} than in Figure~\ref{fig:m9}. It emerges because of an insufficient number of experiments. Indeed, $\widehat{\theta}_T$ materializing just outside of~$\Theta$ appears to be a rare event in high dimensions.}
\begin{figure*}[t!]
    \centering
    \begin{subfigure}[b]{0.22\textwidth}
        \includegraphics[width=\textwidth]{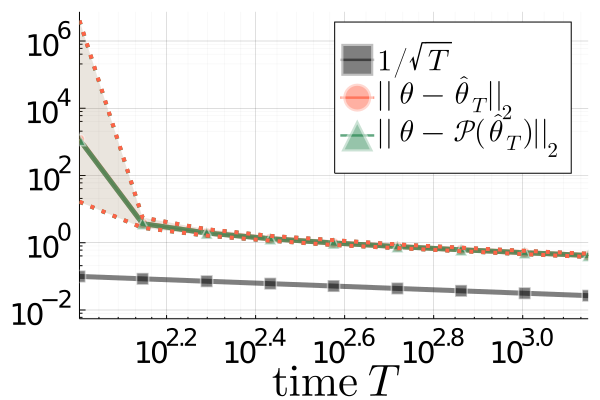}
        \caption{Convergence in operator norm for $n=100$.}
        \label{fig:learn:norm:n100}
    \end{subfigure}\quad
    \begin{subfigure}[b]{0.22\textwidth}
        \includegraphics[width=\textwidth]{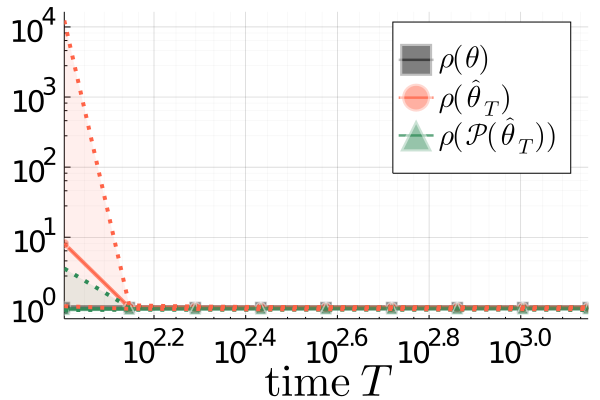}
        \caption{Convergence of spectral radii for $n=100$.}
        \label{fig:learn:rad:n100}
    \end{subfigure}\quad 
    \begin{subfigure}[b]{0.22\textwidth}
        \includegraphics[width=\textwidth]{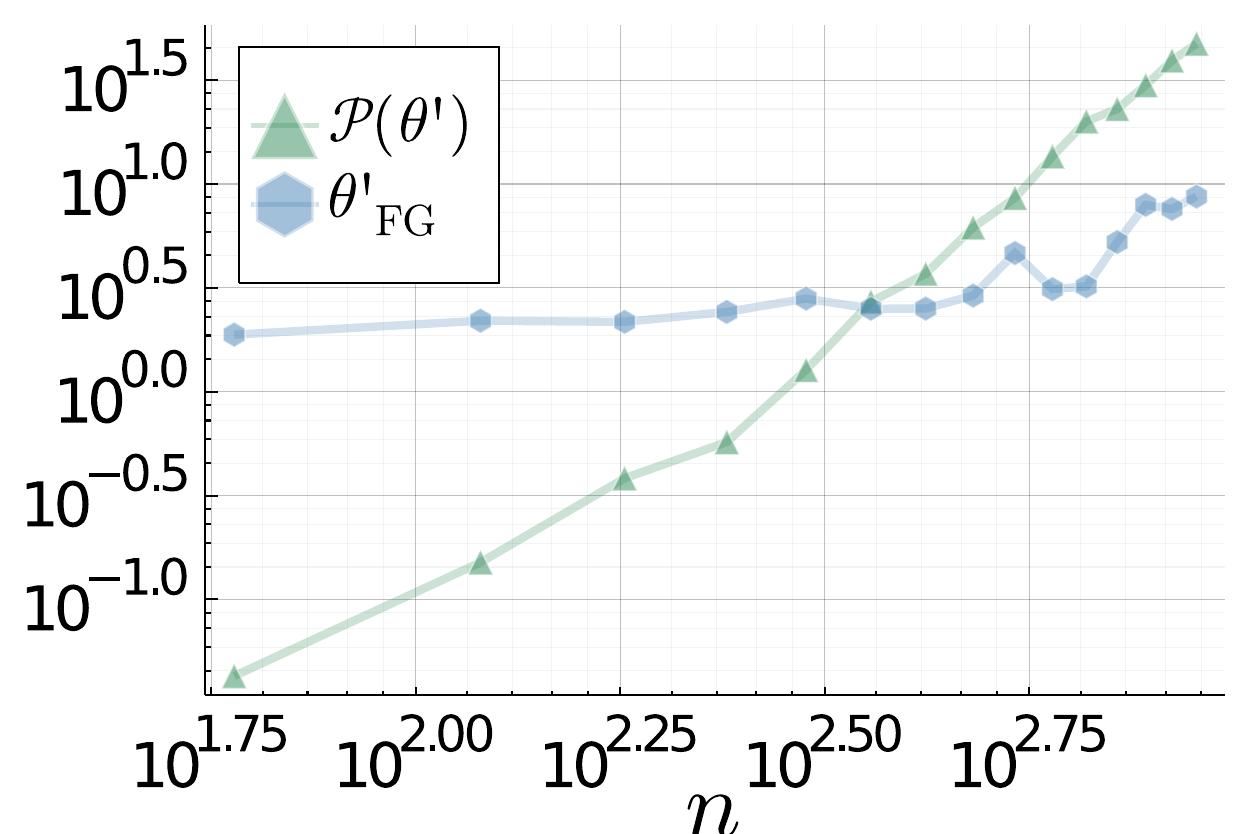}
        \caption{Runtime comparison on $\theta'=(Y\otimes 2 I_3)$.}
        \label{fig:runtime}
    \end{subfigure}\quad
    \begin{subfigure}[b]{0.22\textwidth}
        \includegraphics[width=\textwidth]{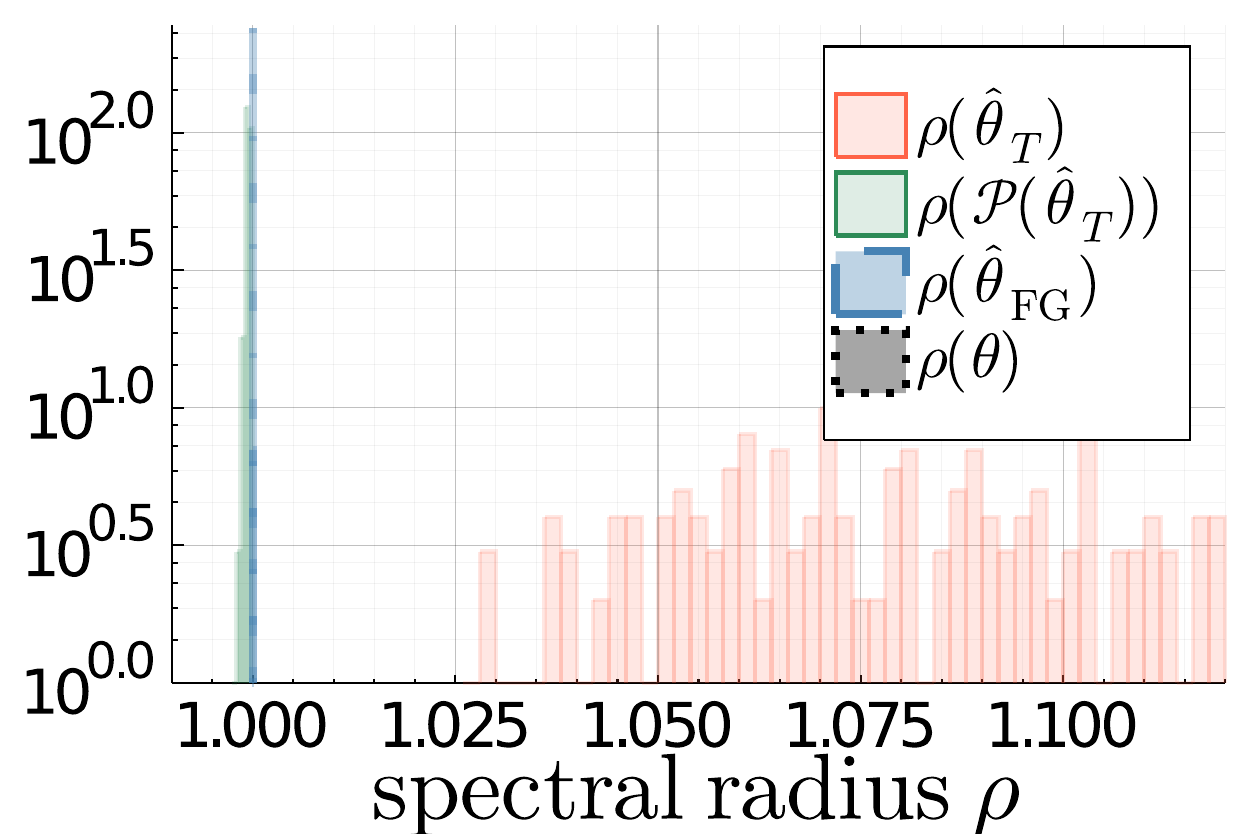}
        \caption{Historgrams of the spectral radii for $n=192$.}
        \label{fig:m64}
    \end{subfigure}
    \caption[]{Additional material related to Section~\ref{sec:num}.}
    \label{fig:learn:stab:random:app}
\end{figure*}

\newpage
\pagestyle{basicstyle}
\subsection*{Bibliography}
\printbibliography[heading=none]


\end{document}

%% file: style.tex
\usepackage[top=4cm, bottom=4cm, inner=3cm,outer=3cm]{geometry}
\usepackage[utf8]{inputenc}
\usepackage{amsthm}
\usepackage{mathtools}
\usepackage{csquotes}
\usepackage{amsmath}
\usepackage{graphicx}
\usepackage{dsfont}
\usepackage{soul}
\usepackage{amssymb}
\usepackage{multicol}
\usepackage{courier}
\usepackage{listings}
\usepackage{eurosym}
\usepackage{subfiles}
\usepackage{multicol}
\usepackage{pdfpages}
\usepackage{footnote}
\usepackage{color}
\definecolor{myblue}{rgb}{.25, .25, .9}
\definecolor{myred}{rgb}{.6, .4, .4}
\definecolor{myred2}{rgb}{.9, .1, .1}
\definecolor{mygreen}{rgb}{.25, .6, .5}
\usepackage{adjustbox}
\usepackage[colorlinks=true,
            linkcolor=myred2,
            urlcolor=myred,
            citecolor=myred]{hyperref}
\usepackage{pdfpages}
\usepackage[english]{babel}
\usepackage{array}
\usepackage{enumerate}
\usepackage{gensymb}
\usepackage[nottoc]{tocbibind}
\usepackage{lettrine}
\usepackage[normalem]{ulem}
\usepackage[font=small,labelfont=bf]{caption}
\usepackage{cancel}
\usepackage{nomencl}
\usepackage{mathrsfs}
\usepackage{changepage}
\usepackage{amsmath,stmaryrd,graphicx,float}
\numberwithin{equation}{section}

\usepackage{tikz-cd}
\usetikzlibrary{arrows}

\usepackage{subcaption} 

\usepackage{titlesec}

\newtheoremstyle{mystyle}
  {}
  {}
  {\itshape}
  {}
  {\bfseries}
  {.}
  { }
  {}

\theoremstyle{mystyle}

\newtheorem{theorem}{Theorem}[section]

\newtheorem{definition}[theorem]{Definition}
\newtheorem{lemma}[theorem]{Lemma}
\newtheorem{proposition}[theorem]{Proposition}
\newtheorem{corollary}[theorem]{Corollary}

\newtheorem{example}[theorem]{Example}

\newtheorem{remark}[theorem]{Remark} 
\newtheorem{assumption}[theorem]{Assumption}

\usepackage{etoolbox}

\DeclareMathAlphabet{\mathbbold}{U}{bbold}{m}{n}

\titleformat{\section}
{\large\filcenter\bfseries}{\llap{\thesection\hskip 9pt}}{0pt}{}
\titlespacing*{\section}
{0pt}{3.25ex plus 1ex minus .2ex}{1.5ex plus .2ex}

\titleformat{\subsection}[runin]
{\bfseries}{\llap{\thesubsection\hskip 9pt}}{0pt}{}
\titlespacing*{\subsection}
{0pt}{3.25ex plus 1ex minus .2ex}{1.5ex plus .2ex}

\titleformat{\subsubsection}[runin]
{\bfseries}{\llap{\thesubsubsection\hskip 9pt}}{0pt}{}
\titlespacing*{\subsubsection}
{0pt}{3.25ex plus 1ex minus .2ex}{1.5ex plus .2ex}

\titleformat{\paragraph}[runin]
{\bfseries}{\llap{\theparagraph\hskip 9pt}}{0pt}{}
\titlespacing*{\paragraph}
{0pt}{3.25ex plus 1ex minus .2ex}{1.5ex plus .2ex}

\usepackage[titles]{tocloft}

\usepackage{fancyhdr}
\usepackage{lastpage}
\fancypagestyle{plain}{
\fancyhf{}
\fancyhead[LE]{\thepage\quad\textbf{\nouppercase\leftmark}}
\fancyhead[RO]{\textbf{\nouppercase\rightmark}\quad\thepage}
}

\fancypagestyle{basicstyle}{%
\fancyhf{}
\fancyhead[LE]{\thepage}
\fancyhead[RO]{\thepage}
}

\newcommand{\bmat}{\begin{pmatrix}}
\newcommand{\emat}{\end{pmatrix}}

\usepackage{algorithm}
\usepackage{algorithmic}[1]

\usepackage{chngcntr}
\counterwithin{figure}{section}

\usepackage[backend=bibtex,
style=authoryear,
bibencoding=ascii,
giveninits=true, 
]{biblatex}
\renewbibmacro{in:}{}

\usepackage{xcolor, calc, cleveref}
\newcommand{\noopsort}[1]{}

\newcommand{\drop}[1]{}

\usepackage{rotating}
\usepackage{mathtools}